\documentclass{article}

\usepackage{amsmath}
\usepackage{amssymb}
\usepackage{amsthm}
\usepackage{enumerate}

\numberwithin{equation}{section}  

\newtheorem{lem}{Lemma}[section]
\newtheorem{thm}[lem]{Theorem}
\newtheorem{dfn}[lem]{Defination}
\newtheorem{rmk}[lem]{Remark}
\newtheorem{prp}[lem]{Proposition}
\newtheorem{cor}[lem]{Corollary}
\newtheorem{exm}[lem]{Example}

\begin{document}

\title{Unstability of pseudoharmonic maps between pseudo-Hermitian manifolds
\author{Tian Chong, Yuxin Dong and Yibin Ren}
\footnote{This work was supported by NSFC grant No. 11271071 and LMNS, Fudan.}}
\date{}
\maketitle
\begin{abstract}
In this paper, we derive the second variation formula of pseudoharmonic maps into any pseudo-Hermitian manifolds. When the target manifold is an isometric embedded CR manifold in complex Euclidean space or a pseudo-Hermitian immersed
submanifold in Heisenberg group, we give some conditions on Weingarten maps to obtain some unstability
of pseudoharmonic maps between these pseudo-Hermitian manifolds.
\end{abstract}

\section{Introduction}

\ \ \ \ As known to all, a harmonic map is a critical point of the energy integral. A harmonic map is called stable if it has nonnegative second variation, that is, the index of the map is 0. The stability problem is an important problem in the theory of harmonic maps. In [7], R.T. Smith estimated the index of the identity map of a Riemannian manifold, in particular, he showed that the identity map on $S^{m}$  is $m+1$. In [8], Y.L. Xin proved that for $m\geq3$, any nonconstant harmonic map $f:S^{m}\rightarrow N^n$ is unstable. A result of Leung [5] states that any nonconstant map from a compact Riemannian manifold to the sphere  is unstable too. In [4], the authors extended the Leung's result to the case that the target manifold is a compact immersed submanifold of Euclidean space.\\
\indent In recent years, some generalized harmonic maps have been introduced and investigated in various geometric backgrounds. For example, some pseudoharmonic maps were introduced in the field of pseudo-Hermitian geometry (cf. [1],[2],[6]). For a map $f:(M^{2m+1},H(M),J,\theta)\rightarrow(N,\widetilde{H}(N),\widetilde J,\widetilde{\theta})$ between two pseudo-Hermian maps, Petit [6] introduced a new horizontal energy functional $E_{H,\widetilde{H}}(f)$. He derived the first variation formula of $E_{H,\widetilde{H}}(f)$  and called a critical point of the energy a pseudoharmonic map. Note that there is an extra condition on the pull-back of the torsion on the target manifold. The second author [1] modified Petit's variational problem slightly by restricting the variational vector field to be horizontal. In [1], the critical point of the restricted variantional problem about   is refered to as a pseudoharmonic too. Among other results, the second author derived the second variation formula of pseudoharmonic maps into Sasakian manifold and proved that any nonconstant horizontal pseudoharmonic map from a closed pseudo-Hermitian manifold into the odd dimensional sphere is unstable.\\
\indent In this paper, we will extend their results to the case that the target manifold is an isometric embedded CR manifold or a pseudo-Hermitian immersed submanifold of Heisenberg group and give a low bound of the index of identity map $I:S^{2n+1}\rightarrow S^{2n+1}$ with $n\geq1$. Firstly, we derive the second variation formula of pseudoharmonic maps into any pseudo-Hermitian manifolds. Then we give a condition on Weingarten map which implies that there is no nonconstant horizontal pseudoharmonic map from a closed pseudo-Hermitian manifold into an isometric embbed CR manifold. Next we consider the identity map $I:S^{2n+1}\rightarrow S^{2n+1}$. From the above result we known that it is unstable.  Following the result of [7] we discuss the degree of the unstability and  derive that index($I$)$\geq$ 2n+2. In the end, we give a condition on CR Weingarten map which also implies there is no nonconstant pseudoharmonic map from a closed pseudo-Hermitian manifold into a pseudo-Hermitian immersed submanifold of Heisenberg group.
\section{Basic Notions}
\ \ \ \ We will follow mainly the notations and terminologies in [2], but a somewhat different wedge multiplication for forms.\\
\indent Let $M$ be a real ($2m+1$)-dimensional $C^{\infty}$ differentiable manifold. Let $TM \otimes \mathbb{C}$ be the complexified tangent bundle over $M$.
Let $T_{1,0}M\subseteq TM \otimes \mathbb{C}$ be a complex subbundle of complex rank $m$.
$T_{1,0}M$ is called an $CR$ structure on $M$, if $T_{1,0}M\cap T_{0,1}M=0 $ and $[T_{1,0}M,T_{1,0}M]\subseteq (T_{1,0}M)$.
Here $T_{0,1}M=\overline{T_{1,0}M}$.\\
\indent A pair ($M,T_{1,0}M$) consisting of a $C^{\infty}$ manifold and a $CR$ structure is a $CR$ manifold. The integer $m$ is the CR dimension. \\
\indent Its Levi distribution is the real rank $2m$ subbundle $H(M)\subseteq TM$ given by $H(M)=Re\{T_{1,0}M\oplus T_{0,1}M\}$. It carries the complex structure $J_{b}:H(M)\rightarrow H(M)$ given by $J_{b}(V+\overline{V})=\sqrt{-1}(V-\overline{V})$ for any $V\in\Gamma(T_{1,0}M)$.
\begin{dfn}\quad
Let ($M,T_{1,0}M$) and ($N,T_{1,0}N$) be two $CR$ manifolds. A $C^{\infty}$ map $f:M\rightarrow N$ is a CR map if
\begin{align}
(d_xf)(T_{1,0}M)_x\subseteq (T_{1,0}N)_{f(x)} ,  \nonumber
\end{align}
for any $x\in M$, where $d_{x}f$ or $(f_{*})_x$ is the differential of $f$ at $x$.
\end{dfn}
\indent There exists a global nonvanishing 1-form $\theta$ such that $ker(\theta_{x})\supseteq H(M)_x$ for any $x\in M$.
Any such section $\theta$ is referred to as a pseudo-Hermitian structure on $M$. Given a pseudo-Hermitian structure $\theta$ on $M$, the Levi-form $L_{\theta}$ is defined by $L_{\theta}(Z,\overline{W})=-\frac{\sqrt{-1}}{2}(d\theta)(Z,\overline{W})$
for any $Z,W \in \Gamma(T_{1,0}M)$.\\
\indent We can define the bilinear form $G_{\theta}$ by setting $G_{\theta}(X,Y)=\frac{1}{2}(d\theta)(X,J_{b}Y)$
for any $X,Y \in \Gamma(H(M))$. Since $L_{\theta}$ and (the $\mathbb{C}$ linear extension to $H(M)\otimes\mathbb{C}$ of) $G_{\theta}$ coincide on $H(M)\otimes H(M)$, then
$G_{\theta}(J_{b}X,J_{b}Y)=G_{\theta}(X,Y)$
for any $X,Y \in \Gamma(H(M))$. In particular, $G_{\theta}$ is symmetric.
\begin{dfn}\quad
Let ($M,T_{1,0}M$) be an orientable $CR$ manifold and $\theta$ a fixed pseudo-Hermitian structure on $M$. We call that $(M,T_{1,0}M)$ is a strictly pseudoconvex CR manifold if its Levi form $L_{\theta}$ is positive definite.
\end{dfn}
\begin{rmk}\quad
In this paper, we only consider the strictly pseudoconvex CR manifolds. ($M,H(M),J_{b},\theta$) is called a pseudo-Hermitian manifold, if ($M,T_{1,0}M$) is a strictly pseudoconvex $CR$ manifold and $\theta$ is the pseudo-Hermitian structure such that $L_{\theta}$ is positive definite.
\end{rmk}
If ($M,H(M),J_b,\theta$) is a pseudo-Hermitian manifold, there exists a unique globally defined nowhere zero tangent vector field $T$ on $M$, which transverse to $H(M)$, satisfying
\begin{align}
\theta(T)=1,T\lrcorner \ d \theta=0. \label{70}
\end{align}
The vector $T$ is referred to as the characteristic direction. Then we have the decomposition: $TM=H(M)\oplus RT$.
Using the decomposition we may extend $G_{\theta}$ to a Riemannian metric $g_{\theta}$ on $M$. Let $g_{\theta}$ be the Riemannian metric given by $g_{\theta}(X,Y)=G_{\theta}(\pi_{H}X,\pi_{H}Y)+\theta(X)\theta(Y)$ for any $X,Y\in \Gamma(H(M))$, where $\pi_{H}:TM\rightarrow H(M)$ is the projection associated with the direct sum decomposition. $g_{\theta}$ is called the Webster metric. In this paper we always write $g_{\theta}$ or $G_{\theta}$ as $ \langle , \rangle $ for simplicity.\\

\indent If $\nabla$ is a linear connection on $M$, we use $T_{\nabla}$ to denote its torsion field. A vector-valued 1-form $\tau$: $TM\rightarrow TM$ is defined by $\tau X=T_{\nabla}(T,X)$ for any $X\in\Gamma(TM)$.\\
\indent Let us extend $J_b$ to a $(1,1)$ tensor field $J$ on $M$ by requiring that $J T=0$.
Then ($M,H(M),J,\theta$) becomes a pseudo-Hermitian manifold. \\
\indent On a pseudo-Hermitian manifold, there is a canonical linear connection preserving the complex structure and the webster metric. Precisely, we have the following theorem:
\begin{thm}([2])\quad \label{80}
Let ($M,H(M),J,\theta$) be a pseudo-Hermitian manifold. Let $T$ be the characteristic direction and $J$ the complex structure in H(M). Let $g_{\theta}$ be the webster metric. There is a unique linear connection on $M$(called the Tanaka-Webster connection) such that
\begin{enumerate}[(i)]
\item The Levi distribution H(M) is parallel with respect to $\nabla$.
\item $\nabla J=0,\nabla g_{\theta}=0$.
\item The torsion $T_{\nabla}$ of $\nabla$ satisfies the following conditions:
\begin{align}
T_{\nabla}(Z,W)=0,    \label{73}
\end{align}
\begin{align}
T_{\nabla}(Z,\overline{W})=2\sqrt{-1}L_{\theta}(Z,\overline{W})T, \label{74}
\end{align}
\begin{align}
\tau J+J \tau=0     \label{75}
\end{align}
for any $Z,W \in \Gamma(T_{1,0}M)$.
\end{enumerate}
\end{thm}
By the Theorem \ref{80} we may conclude that $\nabla T=0\ and\  \nabla \theta=0.$ The vector-valued 1-form $\tau$ on $M$ is called the pseudo-Hermitian torsion of $\nabla$.
\begin{rmk}\quad
The pseudo-Hermitian torsion $\tau$ is $H(M)$-valued. It is self-adjoint with respect to $g_{\theta}$ and   $trace\ \tau=0$ (cf.[2]. page37).\\
\indent Let us set $A(X,Y)=g_{\theta}(\tau X,Y)$. Then $A(X,Y)=A(Y,X)$. In particular, ($M,H(M),J,\theta$) is called a Sasakian manifold, if the pseudo-Hermitian torsion $\tau$ is zero.
\end{rmk}
Since $g_{\theta}$ is a Riemannian metric on $M$, then there exists the Levi-Civita connection of ($M,g_{\theta}$) denoted by $\nabla^{\theta}$. We have the following lemma:
\begin{lem}([2])\quad \label{81}
Let ($M,H(M),J,\theta$) be a pseudo-Hermitian manifold. Let $\nabla$ be the Tanaka-Webster connection. Then the torsion tensor field $T_{\nabla}$ of $\nabla$ is given by
\begin{align}
T_{\nabla}=\theta \wedge \tau+d \theta \otimes T. \label{72}
\end{align}
Moreover, the Levi-Civita connection $\nabla^{\theta}$ of ($M,g_{\theta}$) is related $\nabla$ by
\begin{align}
\nabla^{\theta}=\nabla-(\frac{1}{2}d\theta+A)\otimes T+\tau \otimes \theta+2\theta \odot J. \label{77}
\end{align}
 Here $\odot$ denotes the symmetric tensor product. For instance, $2(\theta \odot J)(X,Y)=\theta(X)JY+\theta(Y)JX $ for $\forall X,Y \in\Gamma(TM)$.
\end{lem}
\begin{exm}(Heisenberg group)\quad. \label{92}
The Heisenberg group $\mathbf{H}_{n}$ is obtained by $\mathbf{C}^n\times \mathbf{R}$ with the group law
\begin{align}
(z,t)\cdot(w,s)=(z+w,t+s+2Im \langle z,w \rangle ). \nonumber
\end{align}
Here $(z,t)=(z^1=x^1+y^1,\cdots,z^n=x^n+y^n,t)$ is the natural coordinates.\\

\indent Let us consider the complex vector fields on $\mathbf{H}_{n}$,
\begin{align}
T_{\alpha}=\frac{\partial}{\partial z^{\alpha}}+\sqrt{-1}\overline{z^{\alpha}}\frac{\partial}{\partial t},\ \  \alpha=1,\cdots,n .\nonumber
\end{align}
Here $\frac{\partial}{\partial z^{\alpha}}=\frac{1}{2}(\frac{\partial}{\partial x^{\alpha}}-\sqrt{-1}\frac{\partial}{\partial y^{\alpha}})$ and $z^{\alpha}=x^{\alpha}+y^{\alpha}$. The $CR$ structure $T_{1,0}\mathbf{H}_{n}$ is spanned by $\{T_1,\cdots,T_n\}$. There is a pseudo-Hermitian structure $\theta$ on $\mathbf{H}_n$ defined by
\begin{align}
\theta=dt+2\sum_{i=1}^{n}(x_idy_i-y_idx_i) \nonumber
\end{align}
such that ($\mathbf{H}_n,H(\mathbf{H}),J,\mathbf{\theta}$) becomes a pseudo-Hermitian manifold. Moreover, it is a Sasakian manifold. See [2] for details.
\end{exm}
Now let us discuss the divergence of a vector field on a pseudo-Hermitian manifold. \\
\indent For a vector field $X\in\Gamma(TM)$, the divergence $div(X)$ can be locally computed as:
\begin{align}
div(X)=trace_{g_{\theta}}(\nabla^{\theta}X)=\sum_{\lambda=1}^{2m} \langle \nabla^{\theta}_{e_{\lambda}}X,e_{\lambda} \rangle + \langle \nabla^{\theta}_{T}X,T \rangle ,
\end{align}
where $\{e_{\lambda}\}_{\lambda=1}^{2m}$ is a local orthonormal frame of $H(M)$.\\
\indent Using (\ref{77}) and $\nabla^{\theta}g_{\theta}=0$, we have
\begin{align}
&divX=\sum_{\lambda=1}^{2m}e_{\lambda} \langle X,e_{\lambda} \rangle -\sum_{\lambda=1}^{2m} \langle X,\nabla_{e_{\lambda}}^{\theta}e_{\lambda} \rangle +T(\theta(X)) \nonumber \\
&=\sum_{\lambda=1}^{2m}e_{\lambda} \langle X,e_{\lambda} \rangle -\sum_{\lambda=1}^{2m} \langle X,\nabla_{e_{\lambda}}e_{\lambda} \rangle +(\frac{1}{2}d\theta+A)(e_{\lambda},e_{\lambda})\theta(X)
+T(\theta(X)) \nonumber
\end{align}
Since trace$\tau=0$, then
\begin{align}
divX=\sum_{\lambda=1}^{2m} \langle \nabla_{e_{\lambda}}X,e_{\lambda} \rangle +T(\theta(X)).\label{78}
\end{align}
In particular, for $X\in\Gamma(H(M))$ the identity (\ref{78}) becomes
\begin{align}
divX=\sum_{\lambda=1}^{2m} \langle \nabla_{e_{\lambda}}X,e_{\lambda} \rangle   \label{90}
\end{align}

\section{Pseudoharmonic map}
\ \ \ \ Assume that $(M,H(M),J,\theta)$ and $(N,\widetilde{H}(N),\widetilde{J},\widetilde{\theta})$ are two pseudo-Hermitian manifolds and M is closed, $dim_{R}M=2m+1\; and \; dim_{R}N=2n+1$. \\
\indent For any smooth map $f:M \rightarrow N$
Petit([6]) introduced the following horizontal energy
$E_{H,\widetilde{H}}(f)=\frac{1}{2}\int_{M}|df_{H,\widetilde{H}}|^{2}dv_{g_{\theta}}$.\ Here $df_{H,\widetilde{H}}=\pi_{\widetilde{H}} \circ df\circ i_{H}$, $\pi_{\widetilde{H}}:TN\rightarrow \widetilde{H} (N)$ is the natural projection and $i _H:H(M)\rightarrow TM$ is the natural inclusion. \\
\indent Let $\nabla$ and $\widetilde {\nabla}$ be the Tanaka-Webster connections on M and N respectively. According to [6], we may define the second fundamental form with respect to the data ($\nabla,\widetilde{\nabla}$)
\begin{align}
\beta(X,Y)=(\widetilde{\nabla}_{X}df)(Y)=\widetilde{\nabla}_X df(Y)-df(\nabla_X Y). \nonumber
\end{align}
Here we also use $\widetilde{\nabla}$ to denote the pull-back connection in $f^{-1}TN$. \\
\indent In [6], Petit derived the first variation formula of the energy $E_{H,\widetilde H}$ and call a critical point of $E_{H,\widetilde H}$ a pseudoharmonic map. Note that there is an extra condition on the pull-back of the pseudo-Hermitian torsion. Dong([1]) modified the variational problem slightly and considered the restricted variational problem by requiring the variational vector fields to be horizontal. Then we have
\begin{prp}\label{97} ([1],cf also [6])\quad
For any horizontal vector field $\nu \in \Gamma(f^{-1} \widetilde{H}(N))$, let $\{f_{t}\}(|t|<\varepsilon)$ with $f_{0}=f$ and $\nu= \frac{\partial f_t}{\partial t}|_{t=0}$ be a one parameter variation . Then we have
\begin{align}
\frac{d E_{H,\widetilde{H}}(f_t)}{d t} |_{t=0}=-\int_{M} \langle \nu,\tau_{H,\widetilde{H}} (f) \rangle dv_{\theta}, \nonumber
\end{align}
where
\begin{align}
\; \tau_{H,\widetilde{H}}(f)= tr_{G_{\theta}}({\beta_{H,\widetilde{H}}+(f^* \widetilde{\theta} \otimes f^*\widetilde{\tau})_{H,\widetilde{H}}}); \nonumber
\end{align}
\begin{align}
\beta_{H,\widetilde{H}}=\pi _ {\widetilde{H}} \circ \beta |_{H(M)\times H(M)}; \nonumber
\end{align}
\begin{align}
(f^* \widetilde{\theta} \otimes f^* \widetilde{\tau})_{H,\widetilde{H}}=\pi_{\widetilde{H}}[f^* \widetilde{\theta} \otimes f^* \widetilde{\tau}]|_{H(M)\times H(M)}. \nonumber
\end{align}
We call $\tau _{H,\widetilde{H}}$ the pseudo-tension field of f.
\end{prp}
\begin{dfn}([1])\quad \label{99}
A $C^{\infty}$ map $f:(M,H(M),J,\theta)\rightarrow(N,\widetilde{H}(N),\widetilde{J},\widetilde{\theta})$ is called a pseudoharmonic map if it is a critical point of $E_{H,\widetilde{H}}$ for any horizontal vector field $\nu \in \Gamma(f^{-1}\widetilde{H}(N))$.
\end{dfn}
\begin{cor}([1])\quad
Let $f:(M,H(M),J,\theta)\rightarrow (N,\widetilde H(N),\widetilde J,\widetilde{\theta})$ be a $C^{\infty}$ map. Then $f$ is pseudoharmonic if and only if $\tau_{H,\widetilde H}=0$, that is,
\begin{align}
tr_{G_{\theta}}\{\beta_{H,\widetilde H}(f)+[(f^*\widetilde{\theta})\otimes(f^{*}\widetilde{\tau})]_{H,\widetilde H}\}=0 . \nonumber
\end{align}
\end{cor}
\begin{rmk}\quad \label{96}
 As we have mentioned, the notion of pseudoharmonic maps is slightly different from that in [6]. In [1], the author has shown that Petit's pseudoharmonic maps coincide with Dong's if the target manifold is Sasakian. Henceforth we will investigate pseudoharmonic maps in the sense of Definition \ref{99}.
\end{rmk}
We end this section by proving a lemma which will be used later.
\begin{lem}\quad \label{82}
Let $f:(M,H(M),J,\theta)\rightarrow (N,\widetilde H(N),\widetilde J,\widetilde{\theta})$ be a $C^{\infty}$ map between two pseudo-Hermitian manifolds. Then
\begin{align}
\widetilde{\nabla}_Xdf(Y)-\widetilde{\nabla}_Ydf(X)&=df([X,Y])+\widetilde{\theta}(df(X))\widetilde{\tau}(df(Y))
-\widetilde{\theta}(df(Y))\widetilde{\tau}(df(X))  \nonumber \\
&+d\widetilde{\theta}(df(X),df(Y))\widetilde T    \label{83}
\end{align}
for any $X,Y \in \Gamma(TM)$.
\end{lem}
\begin{proof}\quad
Let ($U,x^1,\cdots,x^{2m+1}$) and ($V,y^{1},\cdots,y^{2n+1}$) be two local coordinate systems on $M$ and $N$ respectively($f(U)\subseteq V$). It is clear that both sides are $C^{\infty}M$ linear. Since $[\frac{\partial}{\partial x^i},\frac{\partial}{\partial x^j}]=0$, by (\ref{72}) we only need to prove that
\begin{align}
\widetilde{\nabla}_{\frac{\partial}{\partial x^i}}df(\frac{\partial}{\partial x^j})-\widetilde{\nabla}_{\frac{\partial}{\partial x^j}}df(\frac{\partial}{\partial x^i})=T_{\widetilde{\nabla}}(df(\frac{\partial}{\partial x^i}),df(\frac{\partial}{\partial x^j})).
\end{align}
\indent We use $(\frac{\partial}{\partial y^a})_{f}$ to denote the pull-back coordinate system. Then
\begin{align}
&\widetilde{\nabla}_{\frac{\partial}{\partial x^i}}df(\frac{\partial}{\partial x^j})-\widetilde{\nabla}_{\frac{\partial}{\partial x^j}}df(\frac{\partial}{\partial x^i}) \nonumber \\
&=\sum_{a=1}^{2n+1}[\widetilde{\nabla}_{\frac{\partial}{\partial x^i}}\frac{\partial f^a}{\partial x^j}\cdot(\frac{\partial}{\partial y^a})_f-
\widetilde{\nabla}_{\frac{\partial}{\partial x^j}}\frac{\partial f^a}{\partial x^i}\cdot(\frac{\partial}{\partial y^a})_f] \nonumber \\
&=\sum_{a,b=1}^{2n+1}[\frac{\partial f^a}{\partial x^j}\frac{\partial f^b}{\partial x^i}\widetilde{\nabla}_{\frac{\partial}{\partial y^b}}\frac{\partial}{\partial y^a}-\frac{\partial f^a}{\partial x^i}\frac{\partial f^b}{\partial x^j}\widetilde{\nabla}_{\frac{\partial}{\partial y^b}}\frac{\partial}{\partial y^a}]  \nonumber \\
&=\sum_{a,b=1}^{2n+1}\frac{\partial f^a}{\partial x^i}\frac{\partial f^b}{\partial x^j}T_{\widetilde{\nabla}}(\frac{\partial}{\partial y^a},\frac{\partial}{\partial y^b}) \nonumber \\
&=T_{\widetilde{\nabla}}(df(\frac{\partial}{\partial x^i}),df(\frac{\partial}{\partial x^j})).
\end{align}
\end{proof}

\section{The second variation formula}
\ \ \ \ In [1], the author derived the second variation formula for pseudoharmonic maps into Sasakian manifolds. In following, we will derive the second variation formula for pseudoharmonic maps into any pseudo-Hermitian manifolds. \\
\indent Firstly, let $f:(M,H(M),J,\theta)\rightarrow(N,\widetilde{H}(N),\widetilde J,\widetilde{\theta})$ be a pseudoharmonic map between two pseudo-Hermitian manifolds. Assume that M is closed.\\
\indent Given any one-parameter variation $ \{f_t\} ( |t|<\varepsilon)$ with
$f_0=f,V=\frac{\partial f_t}{\partial t}|_{t=0}$, we set $\Phi(\cdot,t)=f_t$ and $V_t=\frac{\partial{f_t}}{\partial{t}}$. Moreover, we require $V_t=\frac{\partial f_t}{\partial t}$ to be horizontal for all t, i.e. $\frac{\partial f_t}{\partial t} \in \Gamma(f_t^{-1}(\widetilde H(N)))$.\\
\indent From the first variation formula in [1], we have
\begin{align}
\frac{dE_{H,\widetilde H}(f_t)}{dt}&=\int_M \langle \widetilde{\nabla}_{\frac{\partial}{\partial t}}d\Phi_{H,\widetilde H}(e_{\lambda}),d\Phi_{H,\widetilde H}(e_{\lambda}) \rangle dV_{\theta} \nonumber \\
&=\int_M[ \langle \widetilde{\nabla}_{e_{\lambda}}d\Phi(\frac{\partial}{\partial t}),d\Phi_{H,\widetilde H}(e_{\lambda}) \rangle +\widetilde{\theta}(d\Phi(\frac{\partial}{\partial t})) \langle \widetilde{\tau}(d\Phi(e_{\lambda})),d\Phi_{H,\widetilde{H}}(e_{\lambda}) \rangle  \nonumber \\
&\ \ \ \ \ \ \ \ \ -\widetilde{\theta}(d\Phi(e_{\lambda})) \langle \widetilde{\tau}(d\Phi(\frac{\partial}{\partial t}),d\Phi_{H,\widetilde H}(e_{\lambda}) \rangle ]dV_{\theta} \nonumber \\
&=\int_M[- \langle d\Phi(\frac{\partial}{\partial t}),(\widetilde{\nabla}_{e_{\lambda}}d\Phi_{H,\widetilde{H}})(e_{\lambda}) \rangle +\widetilde{\theta}(d\Phi(\frac{\partial}{\partial t})) \langle \widetilde{\tau}(d\Phi(e_{\lambda})),d\Phi_{H,\widetilde{H}}(e_{\lambda}) \rangle  \nonumber \\
&\ \ \ \ \ \ \ \ \ -\widetilde{\theta}(d\Phi(e_{\lambda})) \langle \widetilde{\tau}(d\Phi(\frac{\partial}{\partial t})),d\Phi_{H,\widetilde H}(e_{\lambda}) \rangle ]dV_{\theta}. \nonumber
\end{align}
Then we get
\begin{align}
\frac{dE_{H,\widetilde H}(f_t)}{dt}=\int_{M}[- \langle V_t,\tau_{H,\widetilde H}(f_t) \rangle +\widetilde{\theta}(d\Phi(\frac{\partial}{\partial t})) \langle \widetilde{\tau}(d\Phi(e_{\lambda})),d\Phi_{H,\widetilde{H}}(e_{\lambda}) \rangle ]dV_{\theta}. \nonumber
\end{align}
Note that $f$ is pseudoharmonic and $\frac{\partial f_t}{\partial t} \in \Gamma(f_t^{-1}(\widetilde H(N)))$ for all $t$ , one gets
\begin{align}
\frac{d^2E_{H,\widetilde H}(f_t)}{dt^2}|_{t=0}
&=-\frac{d}{dt}\{\int_{M} \langle V_t,\tau_{H,\widetilde H}(f_t) \rangle dV_{\theta}\}|_{t=0}  \nonumber \\
&=-\int_{M} \langle V,\widetilde{\nabla}_{\frac{\partial}{\partial t}}\tau_{H,\widetilde H}(f_t) \rangle |_{t=0}dV_{\theta}. \nonumber
\end{align}
\indent Let $\widetilde{R}$ be the curvature tensor fields of $\widetilde{\nabla}$.
Using (\ref{83}) and $\widetilde{\nabla}\theta=0$
we derive that
\begin{align}
&\widetilde{\nabla}_{\frac{\partial}{\partial t}}
\tau_{H,\widetilde H}(f_t)|_{t=0} \nonumber \\
&=\sum_{\lambda=1}^{2m}\widetilde{\nabla}_{\frac{\partial}{\partial t}}[(\widetilde{\nabla}_{e_{\lambda}}df_{tH,\widetilde H})(e_{\lambda})+(f_t^{*}\widetilde{\theta})(e_{\lambda})\widetilde{\tau}(df_{tH,\widetilde H}(e_{\lambda}))]|_{t=0} \nonumber \\
&=\sum_{\lambda=1}^{2m}(\widetilde{\nabla}_{\frac{\partial}{\partial t}}\widetilde{\nabla}_{e_{\lambda}}d\Phi_{H,\widetilde H})(e_{\lambda})|_{t=0}
+\sum_{\lambda=1}^{2m}\widetilde{\nabla}_{\frac{\partial}{\partial t}}
[\widetilde {\theta}(d\Phi(e_{\lambda}))\widetilde{\tau}(d\Phi_{H,\widetilde H}(e_{\lambda}))]|_{t=0} \nonumber \\
&=\sum_{\lambda=1}^{2m}[\widetilde{R}(\frac{\partial }{\partial t},e_{\lambda})d\Phi_{H,\widetilde H}](e_{\lambda})|_{t=0}
+\sum_{\lambda=1}^{2m}(\widetilde{\nabla}_{e_{\lambda}}\widetilde{\nabla}_{\frac{\partial}{\partial t}}d\Phi_{H,\widetilde H})(e_{\lambda})|_{t=0} \nonumber \\
&+\sum_{\lambda=1}^{2m}[d\widetilde{\theta}(V,df(e_{\lambda}))\widetilde{\tau}(df_{H,\widetilde{H}}(e_{\lambda}))
+\widetilde{\theta}(df(e_{\lambda}))(\widetilde{\nabla}_{V}\widetilde{\tau})(df_{H,\widetilde{H}}(e_{\lambda}))] \nonumber \\
&+\sum_{\lambda=1}^{2m}\widetilde{\theta}(df(e_{\lambda}))\widetilde{\tau}(\widetilde{\nabla}_{e_{\lambda}}V-\widetilde{\theta}(df(e_{\lambda}))\widetilde{\tau}(V)) \nonumber
\end{align}
Note that $G_{\widetilde{\theta}}(\cdot,\cdot)=\frac{1}{2}d\widetilde{\theta}(\cdot,\widetilde J \cdot)$.
Then
\begin{align}
&\widetilde{\nabla}_{\frac{\partial}{\partial t}}
\tau_{H,\widetilde H}(f_t)|_{t=0} \nonumber \\
&=\sum_{\lambda=1}^{2m}[\widetilde{R}(\frac{\partial }{\partial t},e_{\lambda})d\Phi_{H,\widetilde H}](e_{\lambda})|_{t=0}
+\sum_{\lambda=1}^{2m}(\widetilde{\nabla}_{e_{\lambda}}\widetilde{\nabla}_{\frac{\partial}{\partial t}}d\Phi_{H,\widetilde H})(e_{\lambda})|_{t=0} \nonumber \\
&+\sum_{\lambda=1}^{2m}[2 \langle \widetilde JV,df_{H,\widetilde H}(e_{\lambda}) \rangle \widetilde{\tau}(df_{H,\widetilde{H}}(e_{\lambda}))
+\widetilde{\theta}(df(e_{\lambda}))(\widetilde{\nabla}_{V}\widetilde{\tau})(df_{H,\widetilde{H}}(e_{\lambda}))] \nonumber \\
&+\sum_{\lambda=1}^{2m}\widetilde{\theta}(df(e_{\lambda}))\widetilde{\tau}(\widetilde{\nabla}_{e_{\lambda}}V)-\sum_{\lambda=1}^{2m}[\widetilde{\theta}(df(e_{\lambda}))]^2\widetilde{\tau}^2(V) \label{54}
\end{align}
Taking into account the identity
\begin{align}
df(e_{\lambda})=\widetilde{\theta}(df(e_{\lambda}))\widetilde T+df_{H,\widetilde H}(e_{\lambda}), \nonumber
\end{align}
we have
\begin{align}
 \langle \widetilde R(V,df(e_{\lambda}))df_{H,\widetilde H}(e_{\lambda}),V \rangle =\widetilde{\theta}(df(e_{\lambda})) \langle \widetilde R(V,\widetilde T)df_{H,\widetilde H}(e_{\lambda}),V \rangle
+ \langle \widetilde R (V,df_{H,\widetilde H}(e_{\lambda}))df_{H,\widetilde H}(e_{\lambda}),V \rangle . \nonumber
\end{align}
By (1.77) in [2], we deduce that
\begin{align}
 \langle \widetilde R(V,df(e_{\lambda}))df_{H,\widetilde H}(e_{\lambda}),V \rangle &=-\widetilde{\theta}(df(e_{\lambda})) \langle \widetilde S(df_{H,\widetilde H}(e_{\lambda}),V),V \rangle
+ \langle \widetilde R (V,df_{H,\widetilde H}(e_{\lambda}))df_{H,\widetilde H}(e_{\lambda}),V \rangle  \nonumber \\
=-\widetilde{\theta}(df(e_{\lambda}))& \langle (\widetilde{\nabla}_{df_{H,\widetilde H}(e_{\lambda})}\widetilde{\tau})V
-(\widetilde{\nabla}_V\widetilde{\tau})(df_{H,\widetilde H}(e_{\lambda})),V \rangle
+ \langle \widetilde R (V,df_{H,\widetilde H}(e_{\lambda}))df_{H,\widetilde H}(e_{\lambda}),V \rangle  . \label{55}
\end{align}
Here $\widetilde{S}$ is given by $\widetilde{S}(X,Y)=(\widetilde{\nabla}_X \widetilde{\tau})(Y)-(\widetilde{\nabla}_Y \widetilde{\tau})(X)$ for any $X,Y \in \Gamma(TN)$.\\
\indent Let $X \in\Gamma(H(M))$ be (locally)defined by
\begin{align}
X=[\sum_{\lambda=1}^{2m} \langle (\widetilde{\nabla}_{\frac{\partial}{\partial t}}d\Phi_{H,\widetilde H})e_{\lambda},V \rangle e_{\lambda}]|_{t=0}. \nonumber
\end{align}
Then we compute the divergence of $X$.
By the divergence theorem, one deduces from $(\ref{83})$ that
\begin{align}
&\int_{M}\sum_{\lambda=1}^{2m} \langle (\widetilde{\nabla}_{e_{\lambda}}\widetilde{\nabla}_{\frac{\partial}{\partial t}}d\Phi_{H,\widetilde H})(e_{\lambda})|_{t=0},V \rangle dV_{\theta} \nonumber \\
&=-\int_{M}\sum_{\lambda=1}^{2m} \langle (\widetilde{\nabla}_{\frac{\partial}{\partial t}}d\Phi_{H,\widetilde H})(e_{\lambda}), \widetilde{\nabla}_{e_{\lambda}}V \rangle |_{t=0} dV_{\theta}  \nonumber \\\
&=-\int_M\sum_{\lambda=1}^{2m}[ \langle \widetilde{\nabla}_{e_{\lambda}}V,\widetilde{\nabla}_{e_{\lambda}}V \rangle
-\widetilde{\theta}(df(e_{\lambda})) \langle \widetilde{\tau}(V),\widetilde{\nabla}_{e_{\lambda}}V \rangle ]dV_{\theta} \label{56}
\end{align}
It follows from (\ref{54}), (\ref{55}) and (\ref{56}) that
\begin{thm}\quad \label{58}
Let $f:(M,H(M),J,\theta)\rightarrow (N,\widetilde{H} (N),\widetilde J,\widetilde{\theta})$ be a pseudoharmonic map between two pseudo-Hermitian manifolds. Assume that $M$ is closed. Let $ \{f_t\} ( |t| < \varepsilon)$ be a family of maps with
$f_0=f,\frac{\partial f_t}{\partial t} \in \Gamma(f_t^{-1}(\widetilde H(N)))$ for all $t$. Set $V=\frac{\partial f_t}{\partial t}|_{t=0}$ . Then the second variation formula of the energy functional $E_{H,\widetilde H}$ is given by
\begin{align}
\frac{d^2E_{H,\widetilde H}(f_t)}{dt^2}&|_{t=0}
=\int_{M}\sum_{\lambda=1}^{2m}\{ \langle \widetilde{\nabla}_{e_{\lambda}}V,\widetilde{\nabla}_{e_{\lambda}}V \rangle -\widetilde R (V,df_{H,\widetilde H}(e_{\lambda}),V,df_{H,\widetilde H}(e_{\lambda})) \nonumber \\
&+\widetilde{\theta}(df(e_{\lambda})) \langle (\widetilde{\nabla}_{df_{H,\widetilde H}(e_{\lambda})}\widetilde{\tau})V,V \rangle +[\widetilde{\theta}(df(e_{\lambda}))]^2 \langle \widetilde{\tau}(V),\widetilde{\tau}(V) \rangle  \nonumber \\
&-2\widetilde{\theta}(df(e_{\lambda}))[ \langle (\widetilde{\nabla}_{V}\widetilde{\tau})(df_{H,\widetilde H}(e_{\lambda})),V \rangle + \langle \widetilde{\tau}(V),\widetilde{\nabla}_{e_{\lambda}}V \rangle ] \nonumber \\
&-2 \langle df_{H,\widetilde H}(e_{\lambda}),\widetilde{\tau}V \rangle  \langle df_{H,\widetilde H}(e_{\lambda}),\widetilde{J}V \rangle \}dV_{\theta}, \label{59}
\end{align}
where $\{e_{\lambda}\}_{\lambda=1}^{2m}$ is a local orthonormal frame of H(M).
\end{thm}
\begin{dfn}([6])\quad
Let $f:(M,H(M),J,\theta)\rightarrow (N,\widetilde{H}(N),\widetilde{J},\widetilde{\theta})$ be a $C^{\infty}$ map between two pseudo-Hermitian manifolds. We say $f$ is horizontal if
\begin{align}
(d_xf)(H_xM)\subseteq \widetilde{H}_{f(x)}(N),
\end{align}
for any $x\in M$.
\end{dfn}
\begin{cor}\quad\label{86}
Let $f:(M,H(M),J,\theta)\rightarrow (N,\widetilde{H} (N),\widetilde J,\widetilde{\theta})$ be a horizontal pseudoharmonic map (i.e. $f$ is horizontal and pseudohamornic ) between two pseudo-Hermitian manifolds. Assume that $M$ is closed.
Let $ \{f_t\} ( |t|<\varepsilon)$ be a family of maps with
$f_0=f,V=\frac{\partial f_t}{\partial t}|_{t=0}$. Moreover, we require $\frac{\partial f_t}{\partial t}$ to be horizontal for all t, i.e.$\frac{\partial f_t}{\partial t} \in \Gamma(f_t^{-1}(\widetilde H(N)))$. Then the second variation formula of the energy functional $E_{H,\widetilde H}$ is given by
\begin{align}
\frac{d^2E_{H,\widetilde H}(f_t)}{dt^2}|_{t=0}
=&\int_M \sum_{\lambda=1}^{2m}[ \langle \widetilde{\nabla}_{e_{\lambda}}V,\widetilde{\nabla}_{e_{\lambda}}V \rangle -\widetilde R(V,df(e_{\lambda}),V,df(e_{\lambda})) \nonumber \\
&-2 \langle df(e_{\lambda}),\widetilde{\tau}V \rangle  \langle df(e_{\lambda}),\widetilde JV \rangle ]dV_{\theta}.
\end{align}
\end{cor}
Similar to the case of harmonic maps between Riemannian manifolds, we may introduce the notion of stability for pseudo-Hermitian harmonic maps as follows.
\begin{dfn}\quad \label{87}
Let $f:(M,H(M),J,\theta)\rightarrow(N,\widetilde{H}(N),\widetilde J,\widetilde{\theta})$ be a pseudoharmonic map. We say $f$ is stable, if $\frac{d^2}{dt^2}|_{t=0}E_{H,\widetilde{H}}(f_t)\geq0$ for any variation $\{f_t\}|_{t <|\epsilon|}$ with $f_0=f,\nu=\frac{\partial f_t}{\partial t}|_{t=0} \in \Gamma(f^{-1}\widetilde{H}(N))$.
If $\frac{d^2}{dt^2}|_{t=0}E_{H,\widetilde{H}}(f_t) < 0$ for some variation ${f_t}$ with $f_0=f,\nu=\frac{\partial f_t}{\partial t}|_{t=0} \in \Gamma(f^{-1}\widetilde{H}(N))$, $f$ is called unstable.
\end{dfn}
Note that if we require $\frac{\partial f_t}{\partial t} \in \Gamma(f^{-1}\widetilde{H}(N))$ to be horizontal for all $t$, by Theorem \ref{58} we have
\begin{align}
\frac{d^2}{dt^2}|_{t=0}E_{H,\widetilde{H}}(f_t)&=\int_{M}\sum_{\lambda=1}^{2m}\{ \langle \widetilde{\nabla}_{e_{\lambda}}V,\widetilde{\nabla}_{e_{\lambda}}V \rangle -\widetilde R (V,df_{H,\widetilde H}(e_{\lambda}),V,df_{H,\widetilde H}(e_{\lambda})) \nonumber \\
&+\widetilde{\theta}(df(e_{\lambda})) \langle (\widetilde{\nabla}_{df_{H,\widetilde H}(e_{\lambda})}\widetilde{\tau})V,V \rangle +[\widetilde{\theta}(df(e_{\lambda}))]^2 \langle \widetilde{\tau}(V),\widetilde{\tau}(V) \rangle  \nonumber \\
&-2\widetilde{\theta}(df(e_{\lambda}))[ \langle (\widetilde{\nabla}_{V}\widetilde{\tau})(df_{H,\widetilde H}(e_{\lambda})),V \rangle + \langle \widetilde{\tau}(V),\widetilde{\nabla}_{e_{\lambda}}V \rangle ] \nonumber \\
&-2 \langle df_{H,\widetilde H}(e_{\lambda}),\widetilde{\tau}V \rangle  \langle df_{H,\widetilde H}(e_{\lambda}),\widetilde{J}V \rangle \}dV_{\theta}. \label{106}
\end{align}
\indent In following, we will mainly use (\ref{106}) to verify the unstability of a pseudo-harmonic map. For convenience, we denote the right hand of (\ref{106}) by $H_f(V,V)$.

\begin{rmk}\quad \label{93}
When the target manifold N is Sasakian, the second author derived the second variation formula of $E_{H,\widetilde{H}}$ which is given by (cf. [1])
\begin{align}
\frac{d^2E_{H,\widetilde H}(f_t)}{dt^2}&|_{t=0}
=\int_{M}\sum_{\lambda=1}^{2m}[ \langle \widetilde{\nabla}_{e_{\lambda}}V_{\widetilde H},\widetilde{\nabla}_{e_{\lambda}}V_{\widetilde H} \rangle -\widetilde R (V,df_{H,\widetilde H}(e_{\lambda}),V,df_{H,\widetilde H}(e_{\lambda}))]dV_{\theta}. \label{200}
\end{align}
where $\{f_t\}$ is the variation of $f$ corresponding to $V$ and $V_{\widetilde H}$ denotes the horizontal part of $V$.
He used this formula to establish some stability and unstability results for pseudoharmoic maps.
\end{rmk}
\section{Pseudoharmonic maps into isometric embedded CR manifolds}
\ \ \ \ In [1], the author has shown that any nonconstant horizontal pseudoharmonic map from a closed pseudo-Hermitian manifold to the odd dimensional sphere is unstable. In this section we generalize his result to the case that the target manifold is an isometric embedded CR manifold.\\
\\
\indent Let $i:N\hookrightarrow \mathbb{C}^{n+k}(\cong\mathbf{R}^{2n+2k})$ be a real (2n+1)-dimensional submanifold($k\geq1$). If N is a CR manifold whose CR structure is induced from $\mathbb{C}^{n+k}$, i.e.
\begin{equation}
T_{1,0}(N)=T^{1,0}(\mathbb{C}^{n+k}) \cap (TN\otimes \mathbb{C}), \nonumber \\
\end{equation}
then N is referred to as an embedded CR manifold. \\
\indent For some pseudo-Hermitian structure $\widetilde{\theta}$ on $N$, let $g_{\widetilde{\theta}}$ be the Webster metric of $(N,\widetilde{\theta})$ and $g_{can}$ be the canonical metric on $\mathbb{C}^{n+k}(\cong\mathbf{R}^{2n+2k})$. We call $(N,g_{\widetilde{\theta}})$ an isometric embedded CR manifold if $ N$ is an embedded CR manifold and $g_{\widetilde{\theta}}$=$i^* g_{can}$.
Then $(N,\widetilde{H}(N),\widetilde{J},\widetilde{\theta})$ is a pseudo-Hermitian manifold, where $\widetilde{H}=ker\widetilde{\theta}$ and $\widetilde{J}$ is induced from the standard complex structure $\widehat{J}$ of $\mathbb{C}^{n+k}(\cong\mathbf{R}^{2n+2k})$. More precisely,
if X $\in\Gamma({\widetilde{H}(N)})$, we have
\begin{equation}
 \widetilde{J}X=\widehat{J}X.   \label{63}
\end{equation}
Let $\widetilde T$ be the characteristic direction of ($N,\widetilde {\theta}$).
If $T^{\bot}N$ is the normal bundle of N in $\mathbb{C}^{n+k}(\cong\mathbf{R}^{2n+2k})$, there exists a vector field $\xi\in \Gamma(T^{\bot}N)$ such that
\begin{equation}
\widetilde{T}=\widehat{J}\xi |_{N} \label{61}
\end{equation}
\begin{exm}([2],cf also [1])\quad \label{91}
The standard odd-dimensional sphere $i:S^{2n+1}\hookrightarrow \mathbb{C}^{n+1}$ is a isometric embedded CR manifold. Moreover, it is a Sasakian manifold.
\end{exm}

In this section we  always assume that $(N,\widetilde{H}(N),\widetilde J,\widetilde{\theta})$ is an isometric embedded CR manifold.
Let $\widehat{\nabla}$ be the standard flat connection on $ \mathbb{C}^{n+k}(\cong \mathbf{R}^{2n+2k})$, $\widetilde{\nabla}^{\theta}$ the Levi-Civita connection of (N,$g_{\widetilde{\theta}}$) and h the second fundamental form of N in $\mathbb{C}^{n+k}(\cong\mathbf{R}^{2n+2k})$. These are related by
\begin{equation}
 \widehat{ \nabla} _{X} Y = \widetilde{\nabla}_{X}^{\theta}Y+h(X,Y),\label{3}
\end{equation}
where X,Y$\in \Gamma(TN)$. \\
\indent For $\eta \in \Gamma(T^{\perp}N)$ and $X \in \Gamma(TN)$, we can define the Weingarten map $A_{\eta} X$ and the connection $\nabla _X^{\perp} \eta$ in the normal bundle by
\begin{equation}
\widehat{ \nabla} _X \eta =-A_{\eta} X+ \nabla_X^{\perp} \eta.
\end{equation}
The tensors h and A are related by
\begin{equation}
 \langle A_{\eta}X,Y \rangle = \langle h(X,Y),\eta  \rangle , \label{15}
\end{equation}
where X and Y are tangent to N and $\eta$ is normal to N. Obviously, $ h(X,Y)$ is symmetric in X and Y and for each $\eta $ the linear map $A_{\eta}$ is self-adjoint.\\
\\
\indent Let $\{v_{2n+2},\cdots ,v_{2n+2k}\}$ be an orthonormal basis for the normal space $T_{y}^{\bot}N$ to $N$ at $y$. Define a linear map $Q_{y}^{N}$: $T_y N \rightarrow T_{y}N$ by
\begin{equation}
Q_{y}^{N}=\sum_{\alpha=2n+2}^{2n+2k}\{2(\pi_{\widetilde{H}}A_{v_{\alpha}})^2
-trace_{G_{\widetilde{\theta}}}(A_{v_{\alpha}})\cdot A_{v_{\alpha}}+2A_{\xi}^2-4Id\},
\end{equation}
where $trace_{G_{\widetilde{\theta}}}(A_{v_{\alpha}})= \sum _{j=1} ^{2n} \langle A_{v_{\alpha}}(X_{j}),\;X_{j} \rangle $ for some (local)orthonormal frame $\{X_{j}:\; 1\leq j\leq 2n\}$ of $\widetilde{H}(N)$.
The definition of $Q_y ^{N}$ does not depend on the choice of orthonormal basis at $y$. Note that for any $X,Y \in \widetilde{H}_y(N)$ we have $\langle Q^N_yX,Y \rangle = \langle X,Q^N_yY \rangle$ .\\
\indent By the Gauss equation for submanifolds, the curvature tensor of $\widetilde{\nabla}^{\theta}$ is given by
\begin{equation}
\widetilde{R}^{\theta}(X,Y,Z,W)= \langle h(Y,W),h(X,Z) \rangle - \langle h(X,W),h(Y,Z) \rangle , \label{84}
\end{equation}
where $X, Y, Z, W$ $\in \Gamma(TN)$.\\
\indent Let $\widetilde{\nabla}$ be the Tanaka-Webster connection of $(N,\widetilde{H}(N),\widetilde{J},\widetilde{\theta})$ and $\widetilde{\tau}$ be pseudo-Hermitian torsion. The sectional curvature of $\widetilde{\nabla}$ is given by(cf. [2], page 49)
\begin{align}
\widetilde{R}(X,Y,X,Y)
&=\widetilde{R}^{\theta}(X,Y,X,Y)+3 \langle \widetilde{J}X,Y \rangle ^2- \langle \widetilde{\tau}X,Y \rangle ^2
+ \langle \widetilde{\tau}X,X \rangle  \langle \widetilde{\tau}Y,Y \rangle ,\nonumber
\end{align}
where $X,Y \in \Gamma(\widetilde{H}(N))$.
By (\ref{84}) we obtain
\begin{align}
\widetilde{R}(X,Y,X,Y)&= \langle h(X,X),h(Y,Y) \rangle -|h(X,Y)|^2+3 \langle \widetilde{J}X,Y \rangle ^2
- \langle \widetilde{\tau}X,Y \rangle ^2  + \langle \widetilde{\tau}X,X \rangle  \langle \widetilde{\tau}Y,Y \rangle ,\label{8}
\end{align}
where $X,Y \in \Gamma(\widetilde{H}(N))$.\\
\indent Next we give the following lemma which will be useful to us later.
\begin{lem}\quad \label{88}
Let $\xi \in\Gamma(T^{\bot}N)$ be the vector field in (\ref{61}). For any $ X \in \Gamma(\widetilde H(N))$ we have
\begin{enumerate}[(i)]
\item $ \langle A_{\xi}\widetilde T,X \rangle =0$
\item $A_{\xi}X=\widetilde J \widetilde{\tau}X-X$.
\end{enumerate}
\end{lem}
\begin{proof}\quad Firstly, let us set $X,Y=\widetilde T $ in (\ref{3}) to obtain (since$\widetilde{\nabla}_{\widetilde T}^{\theta}\widetilde T=0$)
\begin{align}
\widehat{\nabla}_{\widetilde T}\widetilde T =\widetilde{\nabla}_{\widetilde T}^{\theta}\widetilde T+h(\widetilde T,\widetilde T)
=h(\widetilde T,\widetilde  T). \label{62}
\end{align}
Using (\ref{61}), for any $X \in \Gamma(\widetilde{H}(N))$ we have
\begin{align}
 \langle A_{\xi}\widetilde T,X \rangle =- \langle (\widehat{\nabla}_{\widetilde T}{\xi})^{\top},X \rangle
= \langle \widehat{\nabla}_{\widetilde T}\widehat{J}\widetilde T,X \rangle .\label{64}
\end{align}
\indent Note that $\widehat{\nabla}\widehat{J}=0$ and $\widehat{J}^2=-1$. By (\ref{63}) and (\ref{62}), (\ref{64}) may be written as
\begin{align}
 \langle A_{\xi}\widetilde T,X \rangle
=- \langle \widehat{\nabla}_{\widetilde T}\widetilde T,\widetilde J X \rangle
=- \langle h(\widetilde T,\widetilde T),\widetilde{J}X \rangle .
\end{align}
Since $h(\widetilde T,\widetilde T)\in\Gamma(T^{\bot}N)$ and $\widetilde JX\in\Gamma(TN)$, $(i)$ is proved. \\
\indent Next we may use (\ref{77}) and $\widetilde{\nabla}\widetilde T=0$ to perform the following calculations:
\begin{align}
\widehat{\nabla}_{\widetilde T}X&=\widetilde{\nabla}_{\widetilde T}^{\theta}X+h(\widetilde T,X) \nonumber \\
&=\widetilde{\nabla}_{\widetilde T}X+\widetilde JX+h(\widetilde T,X) \nonumber \\
&=\widetilde{\tau}X+[\widetilde T,X]+\widetilde J X+h(\widetilde T,X).
\end{align}
Using the fact that $\widehat{\nabla}$ is torsion free, the above identity becomes
\begin{align}
\widehat{\nabla}_X \widetilde T=\widehat{\nabla}_{\widetilde T}X-[\widetilde T,X]
=\widetilde{\tau}X+\widetilde JX+h(\widetilde T,X).\label{66}
\end{align}
By (\ref{61}) and $\widehat{\nabla}\widehat{J}=0$ we may write (\ref{66}) as
\begin{align}
-\widehat{\nabla}_X \xi=\widetilde J\widetilde{\tau}X-X+\widehat{J}h(\widetilde T,X); \nonumber
\end{align}
hence
\begin{align}
A_{\xi}X=\widetilde{J}\widetilde{\tau}X-X+[\widehat{J}h(\widetilde T,X)]^{\top}. \label{69}
\end{align}
\indent For any $Y \in\Gamma(\widetilde H(N))$. Since $h(\widetilde T,X)\in\Gamma(T^{\bot}N)\ and\  \widehat{J}^2=-1$, by (\ref{63}) we have
\begin{align}
 \langle \widehat{J}h(\widetilde T,X),Y \rangle =- \langle h(\widetilde T,X), \widetilde JY \rangle =0. \label{68}
\end{align}
\indent On the other hand, Using $\widehat{J}^2=-1$ again, by (\ref{63}), (\ref{61}) and (\ref{15}) a computation shows that
\begin{align}
& \langle \widehat Jh(\widetilde T,X),\widetilde T \rangle = \langle A_{\xi}\widetilde T,X \rangle .
\end{align}
We may conclude that
\begin{align}
 \langle \widehat{J} h(\widetilde T,X),\widetilde T \rangle =0 \label{67}
\end{align}
due to $(i)$.
It follows from (\ref{68}) and (\ref{67}) that $[\widehat{J}h(\widetilde T,X)]^{\top}=0$.
Substitution into (\ref{69}) shows that $ A_{\xi}X= \widetilde J \widetilde{\tau}X-X$.
\end{proof}
\indent Let $a$ be a vector field in $\mathbb{C}^{n+k}(\cong \mathbf{R}^{2n+2k})$. We define a vector field
$a^{\top}$ tangent to N and a vector field $a^{\bot}$ normal to N .\\
\indent According to the decomposition $TN=\widetilde{H}(N)\oplus \mathbf{R}\widetilde T$ we may write $a^{\top}$ as
\begin{align}
a^{\top}=\pi_{\widetilde{H}}a^{\top}+ \langle a,\widetilde T \rangle \widetilde T. \nonumber
\end{align}
Let us set $a_{\widetilde H}=\pi_{\widetilde H}a^{\top}$. Then we have
\begin{align}
a_{\widetilde H}
&=a-a^{\bot}- \langle a,\widetilde{T} \rangle  \widetilde{T}, \label{85}
\end{align}
where $a_{\widetilde H}\in \Gamma(\widetilde{H}(N))$. \\
\\
\indent To prove the main result in this section we start with the following lemma.

\begin{lem}\quad\label{7}
Assume that $a$ is a constant vector field. For the vector field $a_{\widetilde H}$ defined above , we have
\begin{align}
\widetilde{\nabla}_X a_{\widetilde{H}} =A_{a^{\perp}}X- \langle A_{a^{\perp}}X,\widetilde{T} \rangle \widetilde{T}- \langle a,\widetilde{T} \rangle \widetilde{\tau} X- \langle a,\widetilde{T} \rangle \widetilde{J}X \nonumber
\end{align}
for any $ X\in \Gamma (\widetilde{H}(N))$.
\end{lem}

\begin{proof}\quad Note that $G_{\widetilde{\theta}}=\frac{1}{2}d\widetilde{\theta}(\cdot,\widetilde J \cdot)$ and $\widehat{\nabla}a=0$ .
Then using (\ref{77}) and (\ref{66}) we perform the following calculations:
\begin{align}
\widetilde{\nabla}_{X}a_{\widetilde H}&=\widetilde{\nabla}_X^{\theta}a_{\widetilde H} + \frac{1}{2} d \widetilde{\theta}(X,a_{\widetilde H})\widetilde{T}+\widetilde{A}(X,a_{H}) \widetilde{T} \nonumber \\
&=[\widehat{\nabla}_{X}(a-a^{\bot}-  \langle a,\widetilde{T}  \rangle  \widetilde{T})]^{\top}+ \langle \widetilde{J}X,a \rangle  \widetilde{T}+ \langle \widetilde{\tau}X,a \rangle \widetilde{T} \nonumber \\
&=(- \widehat{\nabla}_{X}a^{\bot})^{\top}-X \langle a,\widetilde{T} \rangle  \widetilde{T}- \langle a,\widetilde{T} \rangle (\widehat{\nabla}_X \widetilde{T})^{\top}+ \langle \widetilde{J}X,a \rangle  \widetilde{T}
+ \langle \widetilde{\tau}X,a \rangle \widetilde{T} \nonumber \\
&=A_{a^{\bot}}X- \langle a,\widehat{\nabla}_{X} \widetilde{T} \rangle \widetilde T- \langle a,\widetilde{T} \rangle ( \widetilde{\tau}X+\widetilde{J}X)+ \langle \widetilde{J}X,a \rangle  \widetilde{T}
 + \langle \widetilde{\tau}X,a \rangle \widetilde{T} \nonumber \\
&=A_{a^{\bot}}X- \langle A_{a^{\bot}}X, \widetilde{T} \rangle  \widetilde{T}- \langle a,\widetilde{T} \rangle  \widetilde{\tau}X- \langle a,\widetilde{T} \rangle  \widetilde{J}X. \nonumber
\end{align}
\end{proof}
\indent Now we investigate the unstability of horizontal pseudoharmonic maps from a closed pseudo-Hermitian manifold $(M^{2m+1},H(M),J,\theta)$ into an isometric embedded CR manifold $(N^{2n+1},\widetilde{H} (N), \widetilde{J},\widetilde {\theta} )$ .\\
\\
\indent Let $f$: ($M,H(M),J,\theta $)$\rightarrow$ ($N,\widetilde{H}(N),\widetilde{J},\widetilde{\theta}$) be a horizontal pseudoharmonic map between M and N. Let $a$ be a constant vector field in $\mathbb{C}^{n+k}(\cong\mathbf{R}^{2n+2k})$. We use $\varphi_{t}$($|t|<\epsilon$) to denote the flow or one-parameter group of diffeomorphisms generated by $a_{\widetilde H}$. For the horizontal variational vector field $a_{\widetilde H}$ along $f$, the one-parameter variation is $f_{t}= \varphi_{t}\circ f$ with $\ f_0=f,\ \frac{\partial f_t}{\partial t}|_{t=0}=a_{\widetilde H}$ . Since $a_{\widetilde H}$ is a horizontal vector field, then $\frac{\partial f_t }{\partial t}=\frac{\partial \varphi_{t} }{\partial t}\circ f$ is horizontal for all $t$. \\
\indent By Corollary \ref{86}, the second variation formula can be written in the following form:
\begin{align}
\left. \frac{d^2}{d t^2}\right |_{t=0} E_{H,\widetilde{H}}(f_t)
&=\int_{M} \sum_{\lambda =1}^{2m}[|\widetilde{\nabla}_{df(e_{\lambda})}a_{\widetilde H}|^2-\widetilde R (a_{\widetilde{H}},df(e_{\lambda}),a_{\widetilde {H}},df(e_{\lambda}))\nonumber \\
&\quad -2 \langle df(e_{\lambda}),\widetilde {\tau} a_{\widetilde {H}} \rangle  \langle df(e_{\lambda}),\widetilde{J} a_{\widetilde H} \rangle ]dv_{\widetilde {\theta}} \label{6}
\end{align}
Here $\{e_{\lambda}\}_{\lambda=1}^{2m}$ is a local orthonormal frame of H(M).
\begin{thm}\quad \label{107}
Let $f:(M,H(M),J,\theta)\rightarrow (N,\widetilde{H}(N),\widetilde{J},\widetilde{\theta})$ be a nonconstant horizontal pseudoharmonic map from a closed pseudo-Hermitian manifold into an isometric embedded CR manifold. Assume that $Q_{y}^{N}$ is negative definite on $\widetilde{H}(N)$ at each point $y$ of N(i.e. $ \langle Q^{N}X,X \rangle \; < 0$ for all $X\neq0$ and $X\in\Gamma(\widetilde{H}(N).)$
Then $f$ is unstable.
\end{thm}
\begin{proof}\quad  Let $f: M\rightarrow N$ be a nonconstant horizontal pseudoharmonic map. We consider the horizontal vector field $a_{\widetilde H}$ on N as above . By Definition \ref{87} and using (\ref{6}), we have
\begin{align}
H_{f}(a_{\widetilde H},a_{\widetilde H})&= \frac {d^2 E_{H,\widetilde H}(f_t)} {dt^2}|_{t=0} \nonumber \\
&=\int_{M} \sum_{\lambda=1}^{2m}[|\widetilde{\nabla}_{df(e_{\lambda})}a_{\widetilde H}|^2-\widetilde R (a_{\widetilde{H}},df(e_{\lambda}),a_{\widetilde {H}},df(e_{\lambda}))\nonumber \\
&\quad -2 \langle df(e_{\lambda}),\widetilde {\tau} a_{\widetilde {H}} \rangle  \langle df(e_{\lambda},\widetilde{J} a_{\widetilde H}) \rangle ]dv_{\widetilde {\theta}},
\end{align}
where $\{e_{\lambda} \}_{\lambda=1}^{2m}$ is a local orthonormal frame of $H(M)$. Since $f$ is horizontal and by Lemma \ref{7}, we get
\begin{align}
\widetilde{\nabla}_{df(e_{\lambda})} a_{\widetilde{H}}=A _{a^{\bot}}df(e_{\lambda})-\langle A_{a^{\bot}}df(e_{\lambda}),\widetilde T \rangle  \widetilde T - \langle a,\widetilde T \rangle  \widetilde {\tau} df(e_{\lambda})- \langle a,\widetilde T \rangle \widetilde J df(e_{\lambda}) \nonumber
\end{align}
and thus
\begin{align}
\left. |\widetilde{\nabla}_{df(e_{\lambda})}a_{\widetilde{H}}|^2 \right.
&= \langle A_{a^{\bot}}df(e_{\lambda}),A_{a^{\bot}}df(e_{\lambda}) \rangle - \langle A_{a^{\bot}} df(e_{\lambda}) ,\widetilde{T}  \rangle ^2  \nonumber \\
+  \langle a,\widetilde{T} \rangle ^2& \langle  \widetilde{\tau}^2 df(e_{\lambda}),df(e_{\lambda}) \rangle  + \langle a,\widetilde{T} \rangle ^2|df(e_{\lambda})|^2 \nonumber \\
-2 \langle a,\widetilde{T} \rangle & \langle A_{a^{\bot}}df(e_{\lambda}),\widetilde{\tau}df(e_{\lambda}) \rangle
-2 \langle A_{a^{\bot}}df(e_{\lambda}),
\widetilde{J} df(e_{\lambda}) \rangle  \langle a,\widetilde{T} \rangle   \nonumber \\
+2 \langle a,\widetilde{T} \rangle ^2& \langle \widetilde{\tau} \widetilde{J}df(e_{\lambda}),df(e_{\lambda}) \rangle  \label{9}
\end{align}
Next using (\ref{8}), we derive that
\begin{align}
&\widetilde{R} (a_{\widetilde{H}},df(e_{\lambda}),a_{\widetilde{H}},df(e_{\lambda})) \nonumber \\
&= \langle h(a_{\widetilde H},a_{\widetilde{H}}),h(df(e_{\lambda}),df(e_{\lambda})) \rangle -|h(a_{\widetilde{H}},df(e_{\lambda}))|^2
+3 \langle \widetilde{J}a_{\widetilde H},df(e_{\lambda}) \rangle ^2 \nonumber \\
&- \langle a_{\widetilde{H}},\widetilde{\tau}df(e_{\lambda}) \rangle ^2
+ \langle \widetilde{\tau}a_{\widetilde H},a_{\widetilde{H}} \rangle  \langle \widetilde{\tau}df(e_{\lambda}),df(e_{\lambda}) \rangle  \nonumber \\
&= \langle A_{h(a_{\widetilde {H}},a_{\widetilde {H}})} df(e_{\lambda}),df(e_{\lambda}) \rangle -|h(a_{\widetilde{H}},df(e_{\lambda}))|^2
+3 \langle a_{\widetilde H},\widetilde {J} df(e_{\lambda}) \rangle ^2 \nonumber \\
&- \langle a_{\widetilde{H}},\widetilde{\tau}df(e_{\lambda}) \rangle ^2
+ \langle \widetilde{\tau}a_{\widetilde H},a_{\widetilde{H}} \rangle  \langle \widetilde{\tau}df(e_{\lambda}),df(e_{\lambda}) \rangle  \label{10}
\end{align}
\indent For any fixed point $y$, we choose a real orthonormal basis \{$a_1,\ldots,a_{2n+2k}$\} of $\mathbb{C}^{n+k} \cong  {\mathbf{R}^{2n+2k}}$ such that $\{a_1, \ldots ,a_{2n}\}|_y$ is a basis of $\widetilde {H}_y(N)$ , $a_{2n+1}|_y=\widetilde {T}_y$ and $\{a_{2n+2},\ldots,a_{2n+2k}\}|_y$ is a basis of $T^{\perp}_{y}N$.
Then we use $a_i$ $(i=1,\ldots,2n+2k)$ to construct the vector field $(a_i)_{\widetilde {H}}$ as above. By a direct computation at $y$ and using (\ref{9}), (\ref{10}) and $trace\widetilde{\tau}=0$, we may get
\begin{align}
&\sum_{i=1}^{2n+2k} \langle \widetilde{\nabla}_{df(e_{\lambda})}(a_{i})_{\widetilde{H}},\widetilde{\nabla}_{df(e_{\lambda})}(a_{i})_{\widetilde{H}} \rangle  \nonumber \\
&=
\sum_{i =2n+2}^{2n+2k}[ \langle A_{a_{i}}df(e_{\lambda}),A_{a_{i}}df(e_{\lambda}) \rangle - \langle A_{a_{i}}df(e_{\lambda}),\widetilde{T} \rangle ] \nonumber \\
&+ \langle \widetilde{\tau}^2 df(e_{\lambda}),df(e_{\lambda}) \rangle +|df(e_{\lambda})|^2+2 \langle \widetilde {\tau} \widetilde {J}df(e_{\lambda}),df(e_{\lambda}) \rangle   \label{14}
\end{align}
and
\begin{align}
\sum_{i=1}^{2n+2k}\widetilde {R} ((a_i)_{\widetilde H},df(e_{\lambda}),(a_i)_{\widetilde H},df(e_{\lambda}))
&=\sum_{i=1}^{2n}[ \langle A_{h(a_i,a_i)}df(e_{\lambda}),df(e_{\lambda}) \rangle  \nonumber \\
-|h(a_i,df&(e_{\lambda}))|^2] +3|df(e_{\lambda})|^2-|\widetilde{\tau}df(e_{\lambda})|^2  \label{12}
\end{align}
Using $\widetilde J^2=-1$, a easy calculation shows
\begin{align}
\sum_{i=1}^{2n+2k}-2 \langle df(e_{\lambda}),\widetilde{\tau}(a_i)_{\widetilde{H}} \rangle  \langle df(e_{\lambda}),\widetilde{J}(a_i)_{\widetilde H} \rangle
=2 \langle df(e_{\lambda}),\widetilde{\tau} \widetilde {J}df(e_{\lambda}) \rangle   \label{13}
\end{align}
It follows from (\ref{14}),(\ref{12}) and (\ref{13}) that
\begin{align}
&\sum_{i=1}^{2n+2k} H_{f} ((a_i)_{\widetilde H},(a_i)_{\widetilde H}) \nonumber \\
&=\int_M \sum_{\lambda=1}^{2m}\{\sum_{i=2n+2}^{2n+2k}[ \langle \pi_{\widetilde{H}}A_{a_i}df(e_{\lambda}),\pi_{\widetilde{H}}A_{a_i}df(e_{\lambda}) \rangle ]
-\sum_{j=1}^{2n} \langle A_{h(a_j,a_j)}df(e_{\lambda}),df(e_{\lambda}) \rangle  \nonumber \\
&\ \ \ \ +\sum_{j=1}^{2n} \langle h(a_j,df(e_{\lambda})),h(a_j,df(e_{\lambda})) \rangle ]
+2 \langle \widetilde {\tau}^2df(e_{\lambda}),df(e_{\lambda}) \rangle
-2|df(e_{\lambda})|^2+4 \langle \widetilde{\tau} \widetilde{J} df (e_{\lambda}),df(e_{\lambda}) \rangle \}dv_{\theta} \label{17}
\end{align}
By (\ref{15}) and the choice of the basis, we obtain
\begin{align}
\sum_{j=1}^{2n} \langle h(a_j,df(e_{\lambda})),h(a_j,df(e_{\lambda})) \rangle
&=\sum_{j=1}^{2n}\sum_{i=2n+2}^{2n+2k} \langle A_{a_i}df(e_{\lambda}),a_j \rangle ^2
=\sum_{i=2n+2}^{2n+2k} \langle \pi_{ \widetilde{H}} A_{a_i}df(e_{\lambda}),\pi_{ \widetilde{H}} A_{a_i}df(e_{\lambda}) \rangle . \label{18}
\end{align}
Using (\ref{15}) again, we perform the following calculations:
\begin{align}
\sum_{j=1}^{2n} \langle A_{h(a_j,a_j)}df(e_{\lambda}),df(e_{\lambda}) \rangle
&=\sum_{j=1}^{2n} \sum_{i=2n+2}^{2n+2k} \langle A_{a_i} a_j,a_j \rangle  \langle A_{a_i} df(e_{\lambda}),df(e_{\lambda}) \rangle  \nonumber \\
&=\sum_{i=2n+2}^{2n+2k} \langle (trace_{G_{\widetilde {\theta}}}A_{a_i}) A_{a_i}df(e_{\lambda}),df(e_{\lambda}) \rangle    \label{19}
\end{align}
By $(ii)$ in Lemma \ref{88} and (\ref{75}) ,we have
\begin{align}
\widetilde {\tau}^2X=(\widetilde {\tau} \widetilde{J})\cdot(\widetilde{\tau} \widetilde {J})X
&=(-A_{\xi}-Id)\cdot(-A_{\xi}-Id)X \nonumber \\
&=(A_{\xi}^2+2A_{\xi}+Id)X  \label{16}
\end{align}
for any $X \in \Gamma(\widetilde H(N))$.
It follows from (\ref{16}) that
\begin{align}
&2 \langle \widetilde {\tau}^2 df(e_{\lambda}),df(e_{\lambda}) \rangle -2|df(e_{\lambda})|^2+4 \langle \widetilde {\tau} \widetilde{J}df(e_{\lambda}),df(e_{\lambda}) \rangle    \nonumber \\
&=2 \langle A_{\xi}^2 df(e_{\lambda})+2A_{\xi}df(e_{\lambda})+df(e_{\lambda}),df(e_{\lambda}) \rangle -2|df(e_{\lambda})|^2 \nonumber \\
& \: \;\;\;-4 \langle A_{\xi} df(e_{\lambda}),df(e_{\lambda}) \rangle -4|df(e_{\lambda})|^2 \nonumber \\
&= \langle (2A_{\xi}^2-4Id)df(e_{\lambda}),df(e_{\lambda}) \rangle  \label{20}
\end{align}
Finally substituting (\ref{18}), (\ref{19}), (\ref{20}) into (\ref{17}), we get
\begin{align}
\sum_{i=1}^{2n+2k}H_f((a_i)_{\widetilde H},(a_i)_{\widetilde H})=\int_M \sum_{\lambda=1}^{2m} \langle Q^N(df(e_{\lambda})),df(e_{\lambda}) \rangle dv_{\theta} \label{101}
\end{align}
Under the assumption, $\int_M \sum_{\lambda=1}^{2m} \langle Q^N(df(e_{\lambda})),df(e_{\lambda}) \rangle dv_{\theta}$ is negative. We see that at least one $H_f((a_i)_{\widetilde H},(a_i)_{\widetilde H})$ must be negative. Then by Definition \ref{87}, $f$ is unstable.
\end{proof}
\begin{rmk}\quad \label{103}
By (\ref{101}), we observe that the condition $ \langle Q^{N}X,X \rangle < 0$ can be relaxed in such a way that
\begin{align}
\sum_{\lambda=1}^{2m} \langle Q^N(df(e_{\lambda})),df(e_{\lambda}) \rangle  < 0. \nonumber
\end{align}
Here $\{e_{\lambda}\}_{\lambda=1}^{2m}$ is a local orthonormal frame of $H(M)$.
\end{rmk}
Using Theorem \ref{107}, we may recapture the result in [1] as follows:
\begin{cor}(cf.[1])\quad\label{98}
Suppose $f:(M,H(M),J,\theta)\rightarrow(S^{2n+1},\widetilde{H},\widetilde J, \widetilde{\theta})$ is a nonconstant horizontal pseudoharmonic map from a closed pseudo-Hermitian manifold to the odd dimensional sphere. Then $f$ is unstable.
\end{cor}
Before we proof the corollary, we need the following lemma.
\begin{lem}\quad\label{104}
Suppose $f:(M,H(M),J,\theta)\rightarrow (N,\widetilde{H}(N),\widetilde{J},\widetilde{\theta})$ is a horizontal map. If $f$ is nonconstant, then it is horizontally nonconstant.
\end{lem}
\begin{proof}\quad we only have to show that if $df\mid_{H(M)}=0$ then $df\mid_{TM}=0$. For any $X,Y \in \Gamma(H(M))$, since $df\mid_{H(M)}=0$, by (\ref{83}) we have
\begin{align}
(\widetilde{\nabla}_{X}df)(Y)-(\widetilde{\nabla}_{Y}df)(X)=-d \theta(X,Y)df(T). \nonumber
\end{align}
As the levi distribution $H(M)$ is parallel with respect to $\nabla$, we get $df(T)=0$.
\end{proof}
\indent Now we are ready to proof the Corollary \ref{98}.
Let $N=S^{2n+1}$. By Example \ref{91} we know that the standard odd-dimensional sphere $i:S^{2n+1}\hookrightarrow C^{n+1}$ is a isometric embedded CR manifold. Therefore we have $\widetilde{T}=\widehat{J}\nu$. Here $\nu$ be the exterior unit normal to $S^{2n+1}$. It is known that $A_{\nu}X=X$ for all
$X\in \Gamma(TS^{2n+1})$. Then $Q_y^{S^{2n+1}}=2\pi_{\widetilde{H}}-(2n+2)Id $.
Since $f$ is nonconstant, by Lemma \ref{104} we can get $\sum_{\lambda=1}^{2m}|df(e_{\lambda})|^2 > 0$. It follows that $\sum_{\lambda=1}^{2m} \langle Q^N(df(e_{\lambda})),df(e_{\lambda}) \rangle  < 0$. By Remark \ref{103}, we will obtain the result.
\\
\\
\indent Next, we will consider the simplest map, i.e. the identity map $I:S^{2n+1}\rightarrow S^{2n+1}$. Obviously $I$ is a horizontal pseudoharmonic map. By Corallary \ref{98}, we know that $I$ is unstable. Following the method in [7], we want to investigate the unstable degree of $I$ as a pseudoharmonic map.\\
\indent Let $V \in \Gamma(TS^{2n+1})$. According to the decomposition
$TS^{2n+1}=span\{\widetilde{T}\}\oplus \widetilde{H}(S^{2n+1})$, we may write $V$ as
\begin{align}
V=V_{\widetilde T}+V_{\widetilde{H}}, \label{43}
\end{align}
where $V_{\widetilde{T}}= \langle V,\widetilde{T} \rangle  \widetilde{T}$ and $V_{\widetilde H}=V- \langle V,\widetilde{T} \rangle  \widetilde{T}$.
Obviously $V_{\widetilde H}$ is a section of $\widetilde{H} (S^{2n+1})$. It is also a section of $I^{-1}(\widetilde H (S^{2n+1}))$.\\
\indent Let $L_{V}g_{\widetilde{\theta}}$ be the Lie derivative of the webster metric $g_{\widetilde{\theta}}$ in the direction of $V$. It is the symmetric 2-tensor on $S^{2n+1}$.
Firstly, we give the following lemma:
\begin{lem}\quad \label{46}
For any $\ V \in \Gamma(TS^{2n+1})$.
\begin{align}
\frac{1}{2} \int_{S^{2n+1}}|L_V g_{\widetilde{\theta}}|^2_{\widetilde{H}}dV_{\widetilde{\theta}}=
\int_{S^{2n+1}} \{\sum_{\lambda=1}^{2n} &[  \langle  \widetilde{\nabla}_{e_{\lambda}}V_{\widetilde{H}},\widetilde{\nabla}_{e_{\lambda}}V_{\widetilde{H}} \rangle
- \langle \widetilde{R}(V,e_{\lambda},V,e_{\lambda} \rangle ] \nonumber \\
&+( div V_{\widetilde H})^2-2 \langle \widetilde{\nabla}_{\widetilde T}V_{\widetilde H},\widetilde J V_{\widetilde H} \rangle \}d V_{\widetilde{\theta}}. \nonumber
\end{align}
Here
$|L_{V} g_{\widetilde{\theta}}|_{\widetilde H}^2=\sum_{\lambda,\mu=1}^{2n}[(L_{V}g_{\widetilde{\theta}})(e_{\lambda},e_{\mu})]^2$ and $\{e_{\lambda}\}_{\lambda=1}^{2n}$ is a local orthonormal frame of $\widetilde{H}(S^{2n+1})$.
\end{lem}
\begin{proof}\quad By (\ref{43}), we have
\begin{align}
(L_Vg_{\widetilde{\theta}})(e_{\lambda},e_{\mu})=(L_{V_{\widetilde{H}}}g_{\widetilde{\theta}})(e_{\lambda},e_{\mu})+(L_{V_{\widetilde T}} g_{\widetilde {\theta}})(e_{\lambda},e_{\mu}), \nonumber
\end{align}
where $1\leqslant\lambda,\ \mu\leqslant2n$.
Since $S^{2n+1}$ is a Sasakian manifold and $\widetilde{\nabla} \widetilde{T}=0$, we may perform the following calculations:
\begin{align}
(L_{V_{\widetilde{T}}}g_{\widetilde{\theta}})(e_{\lambda},e_{\mu})&=
- \langle [V_{\widetilde{T}},e_{\lambda}],e_{\mu} \rangle - \langle e_{\lambda},[V_{\widetilde{T}},e_{\mu}] \rangle  \nonumber \\
&=-\widetilde{\theta}(V)[ \langle \widetilde{\nabla}_{\widetilde{T}}e_{\lambda},e_{\mu} \rangle
+ \langle e_{\lambda},\widetilde{\nabla}_{\widetilde{T}}e_{\mu} \rangle ] \nonumber \\
&=0 ; \nonumber
\end{align}
hence
\begin{align}
(L_{V}g_{\widetilde{\theta}})(e_{\lambda},e_{\mu})=(L_{V_{\widetilde{H}}}g_{\widetilde{\theta}})(e_{\lambda},e_{\mu}). \label{48}
\end{align}
Since $S^{2n+1}$ is Sasakian, we deduce that
\begin{align}
\frac{1}{2}|L_{V_{\widetilde{H}}}g_{\widetilde{\theta}}|^2&=\frac{1}{2}\sum_{\lambda,\mu=1}^{2n}[(L_{V_{\widetilde H}}g_{\widetilde{\theta}})(e_{\lambda},e_{\mu})]^2 \nonumber \\
=\sum_{\lambda=1}^{2n}& \langle \widetilde{\nabla}_{e_{\lambda}}V_{\widetilde{H}},\widetilde{\nabla}_{e_{\lambda}}V_{\widetilde{H}} \rangle +
\sum_{\lambda,\mu=1}^{2n} \langle \widetilde{\nabla}_{e_{\lambda}}V_{\widetilde H},e_{\mu} \rangle  \langle e_{\lambda},\widetilde{\nabla}_{e_{\mu}}V_{\widetilde H} \rangle  \label{45}
\end{align}
\indent Let $X,Y \in \Gamma(\widetilde H(S^{2n+1}))$ be (locally) defined by $X=\widetilde{\nabla}_{V_{\widetilde H}}V_{\widetilde H}$
and $Y=\sum_{i=1}^{2n} \langle \widetilde{\nabla}_{e_{\lambda}}V_{\widetilde H},e_{\lambda} \rangle V_{\widetilde H}$.
Let us compute the divergence of $X$ and $Y$ respectively. A calculation shows that
\begin{align}
divX-divY&=\sum_{\lambda=1}^{2n}\widetilde R(V_{\widetilde H},e_{\lambda},V_{\widetilde H},e_{\lambda})+\sum_{\lambda,\mu=1}^{2n}[ \langle \widetilde{\nabla}_{[e_{\lambda},e_{\mu}]}V_{\widetilde H},e_{\lambda} \rangle  \langle V_{\widetilde H},e_{\mu} \rangle  \nonumber \\
&+ \langle V_{\widetilde H},\widetilde{\nabla}_{e_{\lambda}}e_{\mu} \rangle  \langle \widetilde{\nabla}_{e_{\mu}}V_{\widetilde H},e_{\lambda} \rangle - \langle \widetilde{\nabla}_{e_{\lambda}}V_{\widetilde H},\widetilde{\nabla}_{e_{\mu}}e_{\lambda} \rangle  \langle V_{\widetilde{H}},e_{\mu} \rangle  \nonumber \\
&+\sum_{\lambda,\mu=1}^{2n} \langle \widetilde{\nabla}_{e_{\lambda}}V_{\widetilde H},e_{\mu} \rangle  \langle e_{\lambda},\widetilde{\nabla}_{e_{\mu}}V_{\widetilde H} \rangle] -(divV_{\widetilde H})^2 \nonumber
\end{align}
Taking into account (\ref{72}) and $\widetilde{\tau}=0$ we may actually express $[e_{\lambda},e_{\mu}]$ as
\begin{align}
[e_{\lambda},e_{\mu}]=\widetilde{\nabla}_{e_{\lambda}}e_{\mu}-\widetilde{\nabla}_{e_{\mu}}e_{\lambda}
-d\widetilde{\theta}(e_{\lambda},e_{\mu})\widetilde T. \nonumber
\end{align}
Then
\begin{align}
divX-divY&=\sum_{\lambda=1}^{2n}\widetilde R(V_{\widetilde H},e_{\lambda},V_{\widetilde H},e_{\lambda})+\sum_{\lambda,\mu=1}^{2n} \langle \widetilde{\nabla}_{e_{\lambda}}V_{\widetilde H},e_{\mu} \rangle  \langle \widetilde{\nabla}_{e_{\mu}}V_{\widetilde H},e_{\lambda} \rangle  \nonumber \\
&+\sum_{\lambda,\mu=1}^{2n}[ \langle \widetilde{\nabla}_{\widetilde {\nabla}_{e_{\lambda}}e_{\mu}}V_{\widetilde H},e_{\lambda} \rangle  \langle V_{\widetilde H},e_{\mu} \rangle - \langle \widetilde{\nabla}_{\widetilde{\nabla}_{e_{\mu}}e_{\lambda}}V_{\widetilde H},e_{\lambda} \rangle  \langle V_{\widetilde H},e_{\mu} \rangle ] \nonumber \\
&+\sum_{\lambda,\mu=1}^{2n}[ \langle V_{\widetilde{H}},\widetilde{\nabla}_{e_{\lambda}}e_{\mu} \rangle  \langle \widetilde{\nabla}_{e_{\mu}}V_{\widetilde H},e_{\lambda} \rangle - \langle \widetilde{\nabla}_{e_{\lambda}}V_{\widetilde H},\widetilde{\nabla}_{e_{\mu}}e_{\lambda} \rangle  \langle V_{\widetilde{H}},e_{\mu} \rangle ] \nonumber \\
&-\sum_{\lambda,\mu=1}^{2n}d\widetilde{\theta}(e_{\lambda},e_{\mu}) \langle \widetilde{\nabla}_{\widetilde T}V_{\widetilde H},e_{\lambda} \rangle  \langle V_{\widetilde H},e_{\mu} \rangle -(divV_{\widetilde H})^2  \label{44}
\end{align}
On the other hand, we have
\begin{align}
\sum_{\lambda,\mu=1}^{2n} \langle \widetilde{\nabla}_{\widetilde {\nabla}_{e_{\lambda}}e_{\mu}}V_{\widetilde H},e_{\lambda} \rangle  \langle V_{\widetilde H},e_{\mu} \rangle
&=\sum_{\lambda,\mu,\nu=1}^{2n} \langle \widetilde{\nabla}_{e_{\lambda}}e_{\mu},e_{\nu} \rangle  \langle \widetilde{\nabla}_{e_{\nu}}V_{\widetilde H},e_{\lambda} \rangle  \langle V_{\widetilde H},e_{\mu} \rangle  \nonumber \\
&=-\sum_{\lambda,\mu,\nu=1}^{2n} \langle e_{\mu},\widetilde{\nabla}_{e_{\lambda}}e_{\nu} \rangle  \langle \widetilde{\nabla}_{e_{\nu}}V_{\widetilde H},e_{\lambda} \rangle  \langle V_{\widetilde H},e_{\mu} \rangle  \nonumber \\
&=-\sum_{\lambda,\mu=1}^{2n} \langle \widetilde{\nabla}_{e_{\mu}}V_{\widetilde H},e_{\lambda} \rangle  \langle V_{\widetilde H},\widetilde{\nabla}_{e_{\lambda}}e_{\mu} \rangle . \nonumber
\end{align}
Similarly, by a easy calculation we get
\begin{align}
\sum_{\lambda,\mu=1}^{2n} \langle \widetilde{\nabla}_{\widetilde{\nabla}_{e_{\mu}}e_{\lambda}}V_{\widetilde H},e_{\lambda} \rangle  \langle V_{\widetilde H},e_{\mu} \rangle=-\sum_{\lambda,\mu=1}^{2n}\langle \widetilde{\nabla}_{e_{\lambda}}V_{\widetilde{H}},\widetilde{\nabla}_{e_{\mu}}e_{\lambda}     \rangle \langle    V_{\widetilde{H}},e_{\mu}   \rangle . \nonumber
\end{align}
Therefore (\ref{44}) becomes (by $G_{\widetilde{\theta}}(\cdot,\cdot)=\frac{1}{2}d\widetilde{\theta}(\cdot,\widetilde{J}\cdot)$)
\begin{align}
divX-divY&=\sum_{\lambda=1}^{2n}\widetilde R(V_{\widetilde H},e_{\lambda},V_{\widetilde H},e_{\lambda})+\sum_{\lambda,\mu=1}^{2n} \langle \widetilde{\nabla}_{e_{\lambda}}V_{\widetilde H},e_{\mu} \rangle  \langle \widetilde{\nabla}_{e_{\mu}}V_{\widetilde H},e_{\lambda} \rangle  \nonumber \\
&-(divV_{\widetilde H})^2+2 \langle \widetilde{\nabla}_{\widetilde T}V_{\widetilde H},e_{\lambda} \rangle  \langle \widetilde{J}V_{\widetilde H},e_{\lambda} \rangle . \nonumber
\end{align}
Then
\begin{align}
& \sum_{\lambda,\mu=1}^{2n} \langle \widetilde{\nabla}_{e_{\lambda}}V_{\widetilde H},e_{\mu} \rangle  \langle \widetilde{\nabla}_{e_{\mu}}V_{\widetilde H},e_{\lambda} \rangle \nonumber \\
&=divX-divY-\sum_{\lambda=1}^{2n}\widetilde R(V_{\widetilde H},e_{\lambda},V_{\widetilde H},e_{\lambda})
+(divV_{\widetilde H})^2 \nonumber \\
&-2 \langle \widetilde{\nabla}_{\widetilde T}V_{\widetilde H},e_{\lambda} \rangle  \langle \widetilde{J}V_{\widetilde H},e_{\lambda} \rangle . \label{100}
\end{align}
Since $S^{2n+1}$ is Sasakian, we have(cf.[2])
\begin{align}
 \langle \widetilde{R}(\widetilde{T},Y)Z,W \rangle = \langle \widetilde{S}(Z,W),Y \rangle =0. \nonumber
\end{align}
Thus
\begin{align}
\widetilde{R}(V_{\widetilde H},e_{\lambda},V_{\widetilde H},e_{\lambda})=\widetilde{R}(V,e_{\lambda},V,e_{\lambda}). \nonumber
\end{align}
Then (\ref{100}) becomes
\begin{align}
& \sum_{\lambda,\mu=1}^{2n} \langle \widetilde{\nabla}_{e_{\lambda}}V_{\widetilde H},e_{\mu} \rangle  \langle \widetilde{\nabla}_{e_{\mu}}V_{\widetilde H},e_{\lambda} \rangle \nonumber \\
&=divX-divY-\sum_{\lambda=1}^{2n}\widetilde R(V,e_{\lambda},V,e_{\lambda})
+(divV_{\widetilde H})^2 \nonumber \\
&-2 \langle \widetilde{\nabla}_{\widetilde T}V_{\widetilde H},e_{\lambda} \rangle  \langle \widetilde{J}V_{\widetilde H},e_{\lambda} \rangle . \nonumber
\end{align}
We may substitute into (\ref{45}) to obtain:
\begin{align}
\frac{1}{2}|L_{V_{\widetilde H}}g_{\widetilde{\theta}}|_{\widetilde H}^2
&=\sum_{\lambda=1}^{2n}[ \langle \widetilde{\nabla}_{e_{\lambda}}V_{\widetilde H},\widetilde{\nabla}_{e_{\lambda}}V_{\widetilde H} \rangle -\widetilde R(e_{\lambda},V,e_{\lambda},V) \nonumber \\
&-2 \langle \widetilde{\nabla}_{\widetilde T}V_{\widetilde H},e_{\lambda} \rangle  \langle \widetilde JV_{\widetilde H},e_{\lambda} \rangle ]
+(divV_{\widetilde H})^2+divX-divY; \nonumber
\end{align}
hence
\begin{align}
\frac{1}{2} \int_{S^{2n+1}}|L_V g_{\widetilde{\theta}}|^2_{\widetilde{H}}dV_{\widetilde{\theta}}=
\int_{S^{2n+1}} \{\sum_{\lambda=1}^{2n}& [  \langle  \widetilde{\nabla}_{e_{\lambda}}V_{\widetilde{H}},\widetilde{\nabla}_{e_{\lambda}}V_{\widetilde{H}} \rangle
- \langle \widetilde{R}(V,e_{\lambda},V,e_{\lambda} \rangle ] \nonumber \\
+( div V_{\widetilde H})^2&-2 \langle \widetilde{\nabla}_{\widetilde T}V_{\widetilde H},\widetilde J V_{\widetilde H} \rangle \}d V_{\widetilde{\theta}}. \nonumber
\end{align}
\end{proof}
\indent By Remark \ref{93} and Lemma \ref{46}, we have the following result.
\begin{prp}\quad
For any $V \in \Gamma(TS^{2n+1})$, we have
\begin{align}
H_I(V,V)=\int_{S^{2n+1}}[\frac{1}{2}|L_Vg_{\widetilde{\theta}}|_{\widetilde H}^2
-(divV_{\widetilde H})^2+2 \langle \widetilde{\nabla}_{\widetilde T}V_{\widetilde H},\widetilde JV_{\widetilde H} \rangle ]dV_{\widetilde{\theta}}, \label{49}
\end{align}
where $V_{\widetilde H}$ denotes the horizontal part of $V$.
\end{prp}
Let $\underline{i}$ denote the algebra of infinitesimal isometries, i.e.vector fields $V$ satisfying $L_Vg_{\widetilde{\theta}}=0$ and let $\underline{c}$ denote the algebra of conformal fields(here we follow the notations of [3]). We have the following conclusion:
\begin{prp}\quad
A vector field $V$ on $S^{2n+1}$ is conformal iff $L_Vg_{\widetilde{\theta}}=\frac{divV_{\widetilde H}}{n}g_{\widetilde{\theta}}$.   \label{51}
\end{prp}
\begin{proof}\quad A vector field is conformal iff $L_Vg_{\widetilde{\theta}}=\sigma g_{\widetilde{\theta}}$, where $\sigma$ is a function. In an orthonormal frame of $\widetilde H(S^{2n+1})$, we see that
\begin{align}
(L_{V}g_{\widetilde{\theta}})(e_{\lambda},e_{\mu})=\sigma g_{\widetilde{\theta}}(e_{\lambda},e_{\mu})=\sigma \delta_{\lambda \mu}. \nonumber
\end{align}
By (\ref{48}), we may conclude that
\begin{align}
(L_{V}g_{\widetilde{\theta}})(e_{\lambda},e_{\mu})=(L_{V_{\widetilde{H}}}g_{\widetilde{\theta}})(e_{\lambda},e_{\mu})
= \langle \widetilde{\nabla}_{e_{\lambda}}V_{\widetilde H},e_{\mu} \rangle + \langle e_{\lambda},\widetilde{\nabla}_{e_{\mu}}V_{\widetilde{H}} \rangle . \nonumber
\end{align}
Therefore
\begin{align}
 \langle \widetilde{\nabla}_{e_{\lambda}}V_{\widetilde H},e_{\mu} \rangle + \langle e_{\lambda},\widetilde{\nabla}_{e_{\mu}}V_{\widetilde{H}} \rangle =\sigma \delta_{\lambda \mu}. \nonumber
\end{align}
Then contract this identity to obtain
\begin{align}
\sigma=\frac{divV_{\widetilde H}}{n}.
\end{align}
\end{proof}
We are now ready to prove the following result:
\begin{prp}\quad
If $n\geq1$, then index($I$)$\geq$ 2n+2. Here the index of $I$ is the dimension of the largest subspace of $\Gamma(TS^{2n+1})$ on which $H_{I}$ is negative.
\end{prp}
\begin{proof}\quad For $V \in \underline{c}$, we have $|L_{V}g_{\widetilde{\theta}}|_{\widetilde H}^2=2\frac{(divV_{\widetilde H})^2}{n}$ so that (\ref{49}) yields
\begin{align}
H_I(V,V)=\int_{S^{2n+1}}\frac{1-n}{n}(divV_{\widetilde H})^2+2 \langle \widetilde{\nabla}_{\widetilde T}V_{\widetilde H},\widetilde JV_{\widetilde H} \rangle dV_{\widetilde{\theta}}   \label{50}
\end{align}
\indent If $V$ is in the orthogonal complement of $\underline{i}$ in $\underline{c}$, it may be seen as restrictions of the constant vector fields on $R^{2n+2}$ by the standard embedding $S^{2n+1}\rightarrow R^{2n+2}$.
It has been shown that(cf. Lemma 6.9 in [1]) $\widetilde{\nabla}_{\widetilde{T}}V_{\widetilde{H}}=-\widetilde{J}V_{\widetilde{H}}$.
We may substitute this identity into (\ref{50}) to get
\begin{align}
H_{I}(V,V)=\int_{S^{2n+1}}[\frac{1-n}{n}(divV_{\widetilde{H}})^2 -2 |V_{\widetilde{H}}| ^2]dV_{\widetilde{\theta}}. \nonumber
\end{align}
Moreover, if $V$ is in the orthogonal complement of $\underline{i}$ in $\underline{c}$, by Proposition \ref  {51} we have
$divV_{\widetilde H}\neq0$ ;
hence
$H_I(V,V)< 0$. Then we obtain index($I$)$\geq$dim$(\underline{c}/ \underline{i})$. i.e. index($I$)$\geq$ 2n+2.
\end{proof}
\section{Pseudoharmonic maps into pseudo-Hermitian submanifolds in Heisenberg groups }
\ \ \ \ In this section, we want to give a condition on the CR weigarten map of a pseudo-Hermitian immersed submanifold $N$ of Heisenberg group which implies that any nonconstant pseudoharmonic map $f$ from a closed pseudo-Hermitian manifold to N is unstable. Firstly we introduce some notions(see [2] for details).\\
\indent Let ($M, H(M), J,\theta$) and $(K,H(K),J_{K},\Theta)$ be two pseudo-Hermitian manifolds of real dimensions $m$ and $m+k$ respectively. We say that a map $f:M\rightarrow K  $ is a CR immersion if $f$ is a CR map and $rank(d_xf)=dimM$ at any $ x\in M$.
\begin{dfn}([2])\quad
Let $f:M\rightarrow K$ be a CR immersion. Then f is called an isopseudo-Hermitian immersion if $f^* \Theta = \theta$ .
\end{dfn}
\begin{rmk}\quad
Using the fact that $f$ is a CR map, we have $J_K \circ f_*=f_* \circ J_M$ and thus $f^*G_{\Theta}=G_{\theta}$.
In general, the immersion $f$ is not isometric(with respect to the Riemannian metrics $g_{\theta}$ and $g_{\Theta}$ ).
\end{rmk}
\indent For simplicity, let's identify $M$ with f(M) and denote the immersion by $i:M \hookrightarrow K$. Since $(K,H(K),J_K,\Theta)$ is a pseudo-Hermitian manifold, ($K,g_{\Theta}$) is a Riemannian manifold.
We can define a vector field $W^{\top}$ tangent to $M$ and a vector field $W^{\bot}$ normal to $N$ by $W^{\top}(x)=tan_xW $ and $W^{\bot}(x)=nor_xW $ for any $x\in M$.  \\
\indent We denote $\nabla ^K$ to be the Tanaka-Webster connection on $K$. Then we set
\begin{align}
\nabla_X Y=(\nabla_{X}^K Y)^{\top},  \label{23} \\
\alpha(X,Y)=(\nabla_{X}^{K}Y)^{\bot}
\end{align}
for any $X \in \Gamma(TM)$ and $\eta$ $\in \Gamma(T^{\bot}M)$.
It is easy to prove that $\nabla$ is a linear connection on $M$, while $\alpha$ is $C^{\infty}(M)$-bilinear and has values in $T^{\bot}M$. Then we obtain the following CR Gauss formula([2]):
\begin{align}
\nabla_{X}^K Y=\nabla _X Y+\alpha(X,Y)
\end{align}
We can also set
\begin{align}
a_{\eta}X=-(\nabla_{X}^K \eta)^{\top}  \\
\nabla _X^{\bot} \eta =(\nabla _{X} ^K \eta )^{\bot}.
\end{align}
for any $X \in \Gamma(TM)$ and $\eta$ $\in \Gamma(T^{\bot}M)$. Then $a$ is $C^{\infty}(M)$-bilinear, while $\nabla ^{\bot}$ is a connection in $T^{\bot}M$. So we have the following CR Wegarten formula([2]):
\begin{align}
\nabla _{X}^K \eta = - a_{\eta}X+ \nabla _X^{\bot} \eta.
\end{align}
The connection $\nabla$ in (\ref{23}) does not coincide with the Tanaka-Webster connection of $(M,\theta)$ in general, nor is $\alpha (f)$ symmetric.\\
\indent Let $i:M\rightarrow K$ be an isopseudo-Hermitian CR immersion between two pseudo-Hermitian manifolds. In [2], it has been proved that $i^*g_{\Theta}=g_{\theta}$ if and only if $T_K^{\bot}=0$.\\
\indent According to the above statement, we have the following definition.
\begin{dfn}([2])\quad \label{25}
A pseudo-Hermitian immersion is an isopseudo-Hermitian CR immersion with the additional property $T_{K}^{\bot}=0$.
\end{dfn}
\begin{rmk}\quad
If $i:M\hookrightarrow K$ is a pseudo-Hermitian immersion, we have $T_M=T_{K}|_{M}$(see details for [2]).
\end{rmk}
Then we have the following theorem.(cf. [2], page354)
\begin{thm}([2])\quad \label{24}
Let ($M,H(M),J,\theta$) and ($K,H(K),J_{K},\Theta$) be two pseudo-Hermitian manifolds and $i:M\rightarrow K $ a pseudo-Hermitian immersion. Then
\begin{enumerate}[(i)]
\item $\nabla$ is the Tanaka-Webster connection of $(M,\theta)$.
\item $\pi _H \alpha \ is \ symmetric $. Here $\pi _{H} \alpha$ be a vector-valued form defined by
\begin{align}
(\pi _H \alpha )(X,Y) =\alpha  (\pi _H X, \pi _H Y) \nonumber
\end{align}
for any $X$,$Y$ $\in TM$.
\item $a_{\eta}$ is $H(M)$-valued and for any $x \in M$, $(a_{\eta})_x: H(M)_x \rightarrow H(M)_x$ is self-adjoint(with respect to $G_{\theta , x}$).
\end{enumerate}
\end{thm}
In fact, the tensors $\alpha$ and $a$ can also related by
\begin{align}
g_{\Theta}(\alpha(X,Y), \eta) =g_{\theta}(a_{\eta}X,Y)  \label{26}
\end{align}
for any $X$,$Y$ $ \in $ $\Gamma(TM)$, $\eta \in \Gamma(T^{\bot}M)$. \\
\\
\indent Next we will consider a special pseudo-Hermitian immersion. Let $K=H_{n+k}$ be the Heisenberg group(with the standard strictly pseudoconvex pseudo-Hermitian structure). By Example \ref{92} we know it is a Sasakian manifold. Let $(N,\widetilde{H}(N),\widetilde{J},\widetilde {\theta})$ be a pseudo-Hermitian manifold of dimension 2n+1.\\
\indent From now on, we always assume that $i:(N,\widetilde{H}(N),\widetilde{J},\widetilde{\theta})\rightarrow H_{n+k}$ is a pseudo-Hermitian immersion.\\
\indent Let $ \overline{\nabla},\widetilde {\nabla}$ be the Tanaka-Webster connections on $\mathbf{H}_{n+k}$ and $N$ respectively. Since $i$ is a pseudo-Hermitian immersion,  by Theorem \ref{24} we have
\begin{align}
\overline{\nabla}_XY
=\widetilde{\nabla}_X Y+\alpha(X,Y)   \label{27}
\end{align}
and
\begin{align}
\overline{\nabla}_X \eta =-a_{\eta}X+\nabla_{X}^{\bot} \eta       \label{28}
\end{align}
for any $X,Y\in\Gamma(TN),\eta \in \Gamma(T^{\bot}N) $.\\
\indent We have the following lemma:
\begin{lem}\quad \label{36}
For any $\eta \in \Gamma(T^{\bot}N)$, $X \in \Gamma(TN)$, $a_{\eta} \widetilde{T}=0$ and $\alpha(\widetilde{T},X)=0$.
Here we still use $\widetilde T$ to denote the characteristic direction on N.
\end{lem}
\begin{proof}\quad For $ X \in \Gamma(TN)$, by (\ref{26}) we have \\
\begin{align}
 \langle a_{\eta} \widetilde{T},X \rangle = \langle \alpha(\widetilde{T},X),\eta \rangle .
\end{align}
Since $\mathbf{H}_{n+k}$ is sasakian, using $T_{\mathbf{H}_{n+k}}|_N=\widetilde{T}, \overline{\nabla} T_{\mathbf{H}_{n+k}}=0,$ and $\widetilde{\nabla} \widetilde{T}=0$ we get
\begin{align}
\alpha(\widetilde{T},X)=(\overline{\nabla}_{\widetilde{T}}X)^{\bot}
=(\overline{\nabla}_X \widetilde{T}+[\widetilde{T},X])^{\bot}
=\overline{\nabla}_X T_{\mathbf{H}_{n+k}}-\widetilde{\nabla}_X\widetilde{T}
=0. \nonumber
\end{align}
\end{proof}
\indent Since $\mathbf{H}_{n+k}$ with the standard strictly pseudoconvex pseudo-Hermitian structure is Sasakian, then $(H,\widetilde{H}(N),\widetilde {J},\widetilde{\theta})$ is also sasakian.
In fact, by Definition \ref{25}, Theorem \ref{24} and $\widetilde{\nabla}\widetilde T=0$ again, we can perform the following calculations:
\begin{align}
\widetilde{\tau}(X) =T_{\widetilde{\nabla}}(\widetilde{T},X)
&=(\overline{\nabla}_{\widetilde{T}}X)^{\top}-[\widetilde{T},X] \nonumber \\
&=(\overline{\nabla}_{T_{\mathbf{H}_{n+k}}|_N}X-[T_{\mathbf{H}_{n+k}}|_N,X])^{\top} \nonumber \\
&=(\tau_{\mathbf{H}_{n+k}}|_N(X))^{\top}    \nonumber \\
&=0
\end{align}
for any $X\in \Gamma(TN)$.
Here $\widetilde{\tau}$ is the pseudo-Hermitian torsion on $N$.\\
\\
\indent For each $x\in N$, let $\{\eta_{2n+2},\cdots,\eta_{2n+2k+1}\}$ be an orthonomal basis for the normal space $T_x^{\bot}N$. According to ($iii$) in Theorem \ref{24}, we can define a selfadjoint linear map $P_x^N:\ \widetilde{H}(N)_x\rightarrow \widetilde H(N)_x$ by
\begin{align}
P_x^N=\sum_{i=2n+2}^{2n+2k+1}[2a_{\eta_{i}}^2-tr_{g_{\widetilde{\theta}}}(a_{\eta_{i}})\cdot a_{\eta_{i}}].
\end{align}
It is easy to see that $P_x^N$ does not depend on the choice of $\{\eta_i\}_{i=2n+2}^{2n+2k+1}$.\\
\\
\indent Let $V$ be a vector in $H_{n+k}$. $V$ can be identified with a parallel vector field on $H_{n+k}(\overline{\nabla}V=0)$ . For any $W\in \Gamma(TN)$, we define a tensor $\mathcal{A}^W$ in $Hom(TN,TN)$ corresponding to $W$ by $\mathcal{A}^W(X)=\widetilde{\nabla}_XW$, for any $X\in \Gamma(TN)$. \\
\indent If we let $W$ be $V^{\top}$, using (\ref{27}), (\ref{28}) and $\overline{\nabla}V=0$ by a direct computation we obtain
\begin{align}
\mathcal{A}^{V^{\top}}(X)=\widetilde{\nabla}_X V^{\top}
=a_{V^{\bot}}(X) \nonumber
\end{align}
and
\begin{align}
\nabla^{\bot}_X V^{\bot}=(\overline{\nabla}_X V^{\bot})^{\bot}
=-\alpha(X,V^{\top}).    \nonumber
\end{align}
These imply
\begin{align}
(\widetilde{\nabla}_{V^{\bot}}\mathcal{A}^{V^{\top}})=(\overline{\nabla}_{V^{\top}}a)_{V^{\bot}}-a_{\alpha(V^{\top},V^{\top})},\label{34}
\end{align}
where $(\overline{\nabla}_Xa)_{V^{\bot}}=(\widetilde{\nabla}_{X}a_{V^{\bot}})-a_{\nabla_X^{\bot}V^{\bot}}$, for any $X \in\Gamma(TN)$.
\begin{lem}\quad  \label{33}
For any $X \in\Gamma(TN)$, \\
\begin{align}
\widetilde{R}(V^{\top},X,V^{\top},X)&=- \langle (\widetilde{\nabla}_{V^{\top}}\mathcal{A}^{V^{\top}})(X),X \rangle - \langle \mathcal{A}^{V^{\top}} \mathcal{A}^{V^{\top}}(X),X \rangle  \nonumber \\
&+ \langle \widetilde{\nabla}_X(\mathcal{A}^{V^{\top}}(V^{\top})),X \rangle
\end{align}
\end{lem}
\begin{proof}\quad
By the definition of the curvature tensor we have
\begin{align}
 \langle \widetilde{R}(V^{\top},X)V^{\top},X \rangle = \langle \widetilde{\nabla}_{V^{\top}}\widetilde{\nabla}_XV^{\top}-\widetilde{\nabla}_X \widetilde{\nabla}_{V^{\top}}V^{\top}-\widetilde{\nabla}_{[V^{\top},X]}V^{\top},X \rangle  \label{31}
\end{align}
\indent Since N is sasakin, by (\ref{72}) we can get
\begin{align}
[V^{\top},X]&=\widetilde{\nabla}_{V^{\top}}X-\widetilde{\nabla}_X V^{\top}-T_{\widetilde{\nabla}}(V^{\top},X) \nonumber \\
&=\widetilde{\nabla}_{V^{\top}}X-\widetilde{\nabla}_X{V^{\top}}-d\widetilde{\theta}(V^{\top},X)\widetilde{T} \nonumber
\end{align}
\indent Putting this into (\ref{31}) , Using Lemma \ref{36} we obtain
\begin{align}
 \langle \widetilde{R}(V^{\top},X)V^{\top},X \rangle
&= \langle \widetilde{\nabla}_{V^{\top}}\widetilde{\nabla}_XV^{\top},X \rangle - \langle \widetilde{\nabla}_X\widetilde{\nabla}_{V^{\top}}V^{\top},X \rangle \nonumber \\
- \langle \widetilde{\nabla}_{\widetilde{\nabla}_{V^{\top}}X}V^{\top},&X \rangle
+ \langle \widetilde{\nabla}_{\widetilde{\nabla}_XV^{\top}}V^{\top},X \rangle +
d\widetilde{\theta} (V^{\top},X) \langle \widetilde{\nabla}_{\widetilde{T}}V^{\top},X \rangle  \nonumber \\
= \langle (\widetilde{\nabla}_{V^{\top}}\mathcal{A}^{V^{\top}})&(X),X \rangle + \langle \mathcal{A}^{V^{\top}}\mathcal{A}^{V^{\top}}(X),X \rangle
- \langle \widetilde{\nabla}_X(\mathcal{A}^{V^{\top}}(V^{\top})),X \rangle . \nonumber
\end{align}
Then we can get
\begin{align}
\widetilde{R}(V^{\top},X,V^{\top},X)&=- \langle (\widetilde{\nabla}_{V^{\top}}\mathcal{A}^{V^{\top}})(X),X \rangle - \langle \mathcal{A}^{V^{\top}} \mathcal{A}^{V^{\top}}(X),X \rangle  \nonumber \\
&+ \langle \widetilde{\nabla}_X(\mathcal{A}^{V^{\top}}(V^{\top})),X \rangle , \nonumber
\end{align}
since $\widetilde{R}(V^{\top},X,V^{\top},X)=- \langle \widetilde R(V^{\top},X)V^{\top},X \rangle $.
\end{proof}
\begin{thm}\quad
Let $i:(N,\widetilde{H}(N),\widetilde J,\widetilde \theta)\rightarrow \mathbf{H}_{n+k}$ be a pseudo-Hermitian immersion and ($M,H(M),J,\theta$) be a closed pseudo-Hermitian manifold. Let $f:(M,H(M),J,\theta)\rightarrow (N,\widetilde{H}(N),\widetilde{J},\widetilde{\theta})$ be a nonconstant pseudoharmonic map. Assume that $P_{y}^{N}$ is negative definite on $\widetilde{H}(N)$ at each point $y$ of N(i.e. for all $X\neq0$ and $X\in\Gamma(\widetilde{H}(N))$,
$ \langle P^{N}X,X \rangle \; < 0$).
Then $f$ is unstable.
\end{thm}
\begin{proof}\quad Firstly we consider the second variation of $f$ corresponding to $V^{\top}$(as above). Since N is sasakian, by Remark \ref{93}, we have
\begin{align}
&H_{f}(V^{\top},V^{\top}) \nonumber \\
&=\int_M\sum_{\lambda=1} ^{2m}[ \langle \widetilde{\nabla}_{df(e_{\lambda})}(V^{\top})_{\widetilde H},\widetilde{\nabla}_{df(e_{\lambda})}(V^{\top})_{\widetilde H} \rangle - \langle \widetilde{R}(V^{\top},
\pi_{\widetilde{H}}df(e_{\lambda}),V^{\top},\pi_{\widetilde H}df(e_{\lambda}) \rangle ]dV_{\theta}, \label{32}
\end{align}
where $\{e_{\lambda} \}_{\lambda=1}^{2m}$ is a local orthonormal frame of $H(M)$.\\
\indent Using $(iii)$ in Theorem \ref{24} and Lemma \ref{36}, the first term in (\ref{32}) may be written as
\begin{align}
&\int_M\sum_{\lambda=1} ^{2m} \langle \widetilde{\nabla}_{df(e_{\lambda})}(V^{\top})_{\widetilde H},\widetilde{\nabla}_{df(e_{\lambda})}(V^{\top})_{\widetilde H} \rangle  \nonumber \\
&=\int_{M}\sum_{\lambda=1}^{2m} \langle  a_{V^{\bot}}(df(e_{\lambda})),a_{V^{\bot}}(df(e_{\lambda})) \rangle  \nonumber \\
&=\int_{M}\sum_{\lambda=1}^{2m} \langle a_{V^{\bot}}(\pi_{\widetilde{H}}(df(e_{\lambda}))),a_{V^{\bot}}(\pi_{\widetilde{H}}(df(e_{\lambda}))) \rangle  \nonumber \\
&=\int_{M}\sum_{\lambda=1}^{2m} \langle a_{V^{\bot}}^2(\pi_{\widetilde{H}}(df(e_{\lambda}))),\pi_{\widetilde{H}}(df(e_{\lambda})) \rangle  \label{40}
\end{align}
\indent By (\ref{34}) and Lemma \ref{33}, the second term in (\ref{32}) may be written as
\begin{align}
&\int_M \sum_{\lambda=1}^{2m}\widetilde{R}(V^{\top},\pi_{\widetilde{H}}df(e_{\lambda}),V^{\top},\pi_{\widetilde H}df(e_{\lambda})) \nonumber \\
&=\int_M \sum_{\lambda=1}^{2m}[- \langle (\widetilde{\nabla}_{V^{\top}}\mathcal{A}^{V^{\top}})(\pi_{\widetilde H}df(e_{\lambda})),\pi_{\widetilde H}df(e_{\lambda}) \rangle - \langle \mathcal{A}^{V^{\top}} \mathcal{A}^{V^{\top}}(\pi_{\widetilde H}df(e_{\lambda})),\pi_{\widetilde H}df(e_{\lambda}) \rangle  \nonumber \\
&\ \ \ \ + \langle \widetilde{\nabla}_{\pi_{\widetilde H}df(e_{\lambda})}(\mathcal{A}^{V^{\top}}(V^{\top})),\pi_{\widetilde H}df(e_{\lambda}) \rangle ]dV_{\theta} \nonumber \\
&=\int _M \sum_{\lambda=1} ^{2m}[- \langle (\overline{\nabla}_{V^{\top}}a)_{V^{\bot}}(\pi_{\widetilde{H}}df(e_{\lambda})),\pi_{\widetilde{H}}df(e_{\lambda}) \rangle
+ \langle a_{\alpha(V^{\top},V^{\top})}(\pi_{\widetilde{H}}df(e_{\lambda})),\pi_{\widetilde{H}}df(e_{\lambda}) \rangle  \nonumber \\
&\ \ \ \ - \langle a_{V^{\bot}}^2(\pi_{\widetilde{H}}(df(e_{\lambda}))),\pi_{\widetilde{H}}(df(e_{\lambda})) \rangle + \langle \widetilde{\nabla}_{\pi_{\widetilde H}df(e_{\lambda})}(\widetilde{\nabla}_{V^{\top}}V^{\top}),\pi_{\widetilde H}df(e_{\lambda}) \rangle ]dV_{\theta}   \label{35}
\end{align}
The last term in (\ref{35}) may be computer as
\begin{align}
&\int_{M} \sum_{\lambda=1}^{2m} \langle \widetilde{\nabla}_{\pi_{\widetilde H}df(e_{\lambda})}(\widetilde{\nabla}_{V^{\top}}V^{\top}),\pi_{\widetilde H}df(e_{\lambda}) \rangle  \nonumber \\
&=\int _M \sum_{\lambda=1}^{2m}[ \langle \widetilde{\nabla}_{df(e_{\lambda})} (\widetilde{\nabla}_{V^{\top}}V^{\top}),\pi_{\widetilde{H}}df(e_{\lambda}) \rangle -\widetilde{\theta}(df(e_{\lambda})) \langle \widetilde{\nabla}_{\widetilde T} (\widetilde{\nabla}_{V^{\top}}V^{\top}),\pi_{\widetilde{H}}df(e_{\lambda}) \rangle ]dV_{\theta} \nonumber \\
&=\int_M \sum_{\lambda=1} ^{2m}[e_{\lambda} \langle \widetilde{\nabla}_{V^{\top}}V^{\top},\pi_{\widetilde{H}}df(e_{\lambda}) \rangle
- \langle \widetilde{\nabla}_{V^{\top}}V^{\top},\widetilde{\nabla}_{e_{\lambda}}(\pi_{\widetilde{H}}df(e_{\lambda})) \rangle  \nonumber \\
&-\widetilde{\theta}(df(e_{\lambda})) \langle \widetilde{\nabla}_{\widetilde T} (\widetilde{\nabla}_{V^{\top}}V^{\top}),\pi_{\widetilde{H}}df(e_{\lambda}) \rangle ]dV_{\theta}. \nonumber
\end{align}
Let $Y\in\Gamma(H(M))$ be locally defined by $Y= \langle \widetilde{\nabla}_{V^{\top}}V^{\top},\pi_{\widetilde{H}}df(e_{\lambda}) \rangle e_{\lambda}$.
Let us computer the divergence of $Y$.
We have
\begin{align}
div(Y)
=\sum_{\lambda=1}^{2m}[e_{\lambda} \langle \widetilde{\nabla}_{V^{\top}}V^{\top},\pi_{\widetilde H}df(e_{\lambda}) \rangle - \langle \widetilde{\nabla}_{V^{\top}}V^{\top},\pi_{\widetilde H}df(\nabla_{e_{\lambda}}e_{\lambda}) \rangle ] \nonumber
\end{align}
Then
\begin{align}
&\int_{M} \sum_{\lambda=1}^{2m} \langle \widetilde{\nabla}_{\pi_{\widetilde H}df(e_{\lambda})}(\widetilde{\nabla}_{V^{\top}}V^{\top}),\pi_{\widetilde H}df(e_{\lambda}) \rangle  \nonumber \\
&=\int_M[- \langle \widetilde{\nabla}_{V^{\top}}V^{\top},tr_{G_{\theta}} \beta_{H,\widetilde{H}} \rangle -\sum_{\lambda=1}^{2m} \widetilde{\theta}(df(e_{\lambda})) \langle \widetilde{\nabla}_{\widetilde T} (\widetilde{\nabla}_{V^{\top}}V^{\top}),\pi_{\widetilde{H}}df(e_{\lambda}) \rangle ]dV_{\theta} \nonumber
\end{align}
Since $f$ is pseudoharmonic, by (\ref{34}) and Lemma \ref{36} we obtain
\begin{align}
&\int_{M} \sum_{\lambda=1}^{2m} \langle \widetilde{\nabla}_{\pi_{\widetilde H}df(e_{\lambda})}(\widetilde{\nabla}_{V^{\top}}V^{\top}),\pi_{\widetilde H}df(e_{\lambda}) \rangle  \nonumber \\
&=-\int_M \sum_{\lambda=1}^{2m}\widetilde{\theta}(df(e_{\lambda})) \langle \widetilde{\nabla}_{\widetilde T} (\widetilde{\nabla}_{V^{\top}}V^{\top}),\pi_{\widetilde{H}}df(e_{\lambda}) \rangle dV_{\theta} \nonumber \\
&=-\int_M \sum_{\lambda=1}^{2m}\widetilde{\theta}(df(e_{\lambda})) \langle (\widetilde{\nabla}_{\widetilde T}\mathcal{A}^{V^{\top}})(V^{\top}),\pi_{\widetilde H}df(e_{\lambda}) \rangle dV_{\theta} \nonumber \\
&=-\int_M \sum_{\lambda=1}^{2m} \widetilde{\theta}(df(e_{\lambda})) \langle (\overline{\nabla}_{\widetilde T}a)_{V^{\bot}}(V^{\top})-a_{\alpha(\widetilde T,V^{\top})}(V^{\top}),\pi_{\widetilde H}df(e_{\lambda}) \rangle dV_{\theta} \nonumber \\
&=-\int_M \sum_{\lambda=1}^{2m} \widetilde{\theta}(df(e_{\lambda})) \langle (\overline{\nabla}_{\widetilde T}a)_{V^{\bot}}(V^{\top}),\pi_{\widetilde H}df(e_{\lambda}) \rangle dV_{\theta} \nonumber
\end{align}
Next, we may substitute into (\ref{35}) to get
\begin{align}
&\int_M \sum_{\lambda=1}^{2m}\widetilde{R}(V^{\top},\pi_{\widetilde{H}}df(e_{\lambda}),V^{\top},\pi_{\widetilde H}df(e_{\lambda})) \nonumber \\
&=\int _M \sum_{\lambda=1} ^{2m}[- \langle a_{V^{\bot}}^2(\pi_{\widetilde{H}}(df(e_{\lambda}))),\pi_{\widetilde{H}}(df(e_{\lambda})) \rangle - \langle (\overline{\nabla}_{V^{\top}}a)_{V^{\bot}}(\pi_{\widetilde{H}}df(e_{\lambda})),\pi_{\widetilde{H}}df(e_{\lambda}) \rangle
 \nonumber \\
&+ \langle a_{\alpha(V^{\top},V^{\top})}(\pi_{\widetilde{H}}df(e_{\lambda})),\pi_{\widetilde{H}}df(e_{\lambda}) \rangle -\widetilde{\theta}(df(e_{\lambda})) \langle (\overline{\nabla}_{\widetilde T}a)_{V^{\bot}}(V^{\top}),\pi_{\widetilde H}df(e_{\lambda}) \rangle ]dV_{\theta} \label{39}
\end{align}
It follows from (\ref{32}),(\ref{40}) and (\ref{39}) that
\begin{align}
&H_{f}(V^{\top},V^{\top})=\left.\frac{d^2}{dt^2}\right |  _{t=0}E_{H,\widetilde{H}}(f_t)\nonumber \\
&=\int _M \sum_{\lambda=1} ^{2m}[2 \langle a_{V^{\bot}}^2(\pi_{\widetilde{H}}(df(e_{\lambda}))),\pi_{\widetilde{H}}(df(e_{\lambda})) \rangle + \langle (\overline{\nabla}_{V^{\top}}a)_{V^{\bot}}(\pi_{\widetilde{H}}df(e_{\lambda})),\pi_{\widetilde{H}}df(e_{\lambda}) \rangle
 \nonumber \\
&- \langle a_{\alpha(V^{\top},V^{\top})}(\pi_{\widetilde{H}}df(e_{\lambda})),\pi_{\widetilde{H}}df(e_{\lambda}) \rangle +\widetilde{\theta}(df(e_{\lambda})) \langle (\overline{\nabla}_{\widetilde T}a)_{V^{\bot}}(V^{\top}),\pi_{\widetilde H}df(e_{\lambda}) \rangle ]dV_{\theta}
\end{align}
\indent We choose an orthonormal basis $\{V_1,\cdots,V_{2n+2k+1}\}$ of $H_{n+k}$ such that $\{V_1,\cdots,V_{2n+1}\}$ is a basis of $T_xN$ and $\{V_{2n+2},\cdots,V_{2n+2k+1}\}$ is a basis of $T_{x}^{\bot}N$.\\
Then $(\overline{\nabla}_{V_{j}^{\top}}a)_{V_{j}^{\bot}}=0$ and $(\overline{\nabla}_{\widetilde T}a)_{V_{j}^{\bot}}(V_{j}^{\top})=0$
as for each $j$ one of $V_{j}^{\top}|_x$ or $V_{j}^{\bot}|_x$ is zero. \\
Then we obtain
\begin{align}
\sum_{j=1}^{2n+2k+1} H_{f}(V^{\top}_j,V^{\top}_j)
=\int _M &\sum_{\lambda=1} ^{2m}[\sum_{i=2n+2}^{2n+2k+1} 2 \langle a_{V_i}^2(\pi_{\widetilde{H}}(df(e_{\lambda}))),\pi_{\widetilde{H}}(df(e_{\lambda})) \rangle  \nonumber \\
&-\sum_{j=1}^{2n+1} \langle a_{\alpha(V_j,V_j)}(\pi_{\widetilde{H}}df(e_{\lambda})),\pi_{\widetilde{H}}df(e_{\lambda}) \rangle ]dV_{\theta}  \label{95}
\end{align}
By (\ref{26}) a calculation shows that
\begin{align}
 \langle a_{\alpha(V_j,V_j)}(\pi_{\widetilde{H}}df(e_{\lambda})),&\pi_{\widetilde{H}}df(e_{\lambda}) \rangle
= \langle \alpha(\pi_{\widetilde{H}}df(e_{\lambda}),\pi_{\widetilde{H}}df(e_{\lambda})),\alpha(V_j,V_j) \rangle  \nonumber \\
&=\sum_{i=2n+2}^{2n+2k+1} \langle \alpha(\pi_{\widetilde{H}}df(e_{\lambda}),\pi_{\widetilde{H}}df(e_{\lambda})),V_i \rangle  \langle \alpha(V_j,V_j),V_i \rangle  \nonumber \\
&=\sum_{i=2n+2}^{2n+2k+1} \langle a_{V_i}(\pi_{\widetilde{H}}df(e_{\lambda})),\pi_{\widetilde{H}}df(e_{\lambda}) \rangle  \langle a_{V_{i}}V_j,V_j \rangle  \nonumber \\
&=\sum_{i=2n+2}^{2n+2k+1} \langle tr_{g_{\widetilde{\theta}}}a_{V_{i}}\cdot a_{V_i}(\pi_{\widetilde{H}}df(e_{\lambda})),\pi_{\widetilde{H}}df(e_{\lambda}) \rangle  \nonumber
\end{align}
Substitute into (\ref{95}), then
\begin{align}
\sum_{j=1}^{2n+2k+1} H_{f}(V^{\top}_j,V^{\top}_j)
=\int _M &\sum_{\lambda=1} ^{2m}\sum_{i=2n+2}^{2n+2k+1} [2 \langle a_{V_i}^2(\pi_{\widetilde{H}}(df(e_{\lambda}))),\pi_{\widetilde{H}}(df(e_{\lambda})) \rangle  \nonumber \\
&- \langle tr_{g_{\widetilde{\theta}}}a_{V_{i}}\cdot a_{V_i}(\pi_{\widetilde{H}}df(e_{\lambda})),\pi_{\widetilde{H}}df(e_{\lambda}) \rangle
]dV_{\theta} \nonumber \\
=\int _M &\sum_{\lambda=1} ^{2m} \langle P^{N}(\pi_{\widetilde{H}}df(e_{\lambda})),\pi_{\widetilde{H}}df(e_{\lambda}) \rangle
dV_{\theta}. \nonumber
\end{align}
Under  the assumption, $P^{N}$ is negative. We see that at least one $H_f(V_j^{\top},V_j^{\top  })$ must be negative. Then $f$ is unstable.
\end{proof}

\vspace{45pt}

Tian Chong \\
School of Mathematical Science, Fudan University\\
Shanghai 200433, China \\
E-mail: valery4619@sina.com \\

Yuxin Dong \\
School of Mathematical Science, Fudan University\\
Shanghai 200433, China \\
E-mail: yxdong@fudan.edu.cn \\

Yibin Ren \\
School of Mathematical Science, Fudan University\\
Shanghai 200433, China \\
E-mail: allenrybqqm@hotmail.com \\

\end{document}